\documentclass{compositio}

\usepackage{amsmath,amsfonts,yhmath,hyperref}
\usepackage{amsthm}
\usepackage{amssymb}
\usepackage[latin9]{inputenc}
\usepackage[nottoc,notlof,notlot]{tocbibind}
\usepackage{mathtools}
\usepackage{enumitem}
\usepackage{tikz-cd}
\usepackage{multirow}
\usepackage{mathrsfs}

\usepackage{pgf,tikz}

\usetikzlibrary{arrows}
\usetikzlibrary[patterns]

\definecolor{qqqqff}{rgb}{0.,0.,1.}
\definecolor{cqcqcq}{rgb}{0.7529411764705882,0.7529411764705882,0.7529411764705882}
\definecolor{ttqqqq}{rgb}{0.2,0.,0.}
\definecolor{qqqqff}{rgb}{0.,0.,1.}
\definecolor{xdxdff}{rgb}{0.49019607843137253,0.49019607843137253,1.}
\definecolor{zzttqq}{rgb}{0.6,0.2,0.}
\definecolor{cqcqcq}{rgb}{0.7529411764705882,0.7529411764705882,0.7529411764705882}

\definecolor{yqyqyq}{rgb}{0.5019607843137255,0.5019607843137255,0.5019607843137255}
\definecolor{uuuuuu}{rgb}{0.26666666666666666,0.26666666666666666,0.26666666666666666}
\definecolor{xdxdff}{rgb}{0.49019607843137253,0.49019607843137253,1.}
\definecolor{qqqqff}{rgb}{0.,0.,1.}

\newcommand{\NN}{\mathbb{N}}
\newcommand{\ZZ}{\mathbb{Z}}
\newcommand{\RR}{\mathbb{R}}
\newcommand{\CC}{\mathbb{C}}

\newcommand{\TT}{\mathbb{T}}

\renewcommand{\P}{\mathcal{P}}
\renewcommand{\L}{\mathcal{L}}
\newcommand{\N}{\mathcal{N}}
\newcommand{\M}{\mathcal{M}}

\newcommand{\R}{\mathcal{R}}

\renewcommand{\S}{\mathcal{S}}

\renewcommand{\O}{\mathcal{O}}

\newcommand{\CCC}{\mathscr{C}}
\newcommand{\EEE}{\mathscr{E}}
\newcommand{\SSS}{\mathscr{S}}

\newcommand{\dzdw}{\frac{\dd z}{z}\wedge\frac{\dd w}{w}}

\newcommand{\dd}{ \mathrm{d} }

\newcommand{\im}{\mathfrak{Im}}

\newtheorem{theo}{Theorem}[section]
\newtheorem*{theom}{Theorem}
\newtheorem{prop}[theo]{Proposition}
\newtheorem{coro}[theo]{Corollary}

\theoremstyle{definition}
\newtheorem{defi}[theo]{Definition}

\theoremstyle{remark}
\newtheorem{remark}[theo]{Remark}
\newenvironment{rem}[1]{
    \begin{remark}#1}{
    \xqed{\blacklozenge}\end{remark}
}

\theoremstyle{remark}
\newtheorem{example}[theo]{Example}
\newenvironment{expl}[1]{
    \begin{example}#1}{
    \xqed{\lozenge}\end{example}
}

\newcommand{\xqed}[1]{
    \leavevmode\unskip\penalty9999 \hbox{}\nobreak\hfill
    \quad\hbox{\ensuremath{#1}}}

\keywords{Enumerative geometry, tropical refined invariants, abelian surfaces, floor diagrams\\ \textit{Data Statement:} I do not have any data to point.}
 
\begin{document}
 
 
\title{Tropical curves in abelian surfaces II:\\ enumeration of curves in linear systems}
\author{Thomas Blomme}

\begin{abstract}
In this paper, second installment in a series of three, we give a correspondence theorem to relate the count of genus $g$ curves in a fixed linear system in an abelian surface to a tropical count. To do this, we relate the linear system defined by a complex curve to certain integrals of $1$-forms over cycles in the curve. We then give an expression for the tropical multiplicity provided by the correspondence theorem, and prove the invariance for the associated refined multiplicity, thus introducing refined invariants of Block-G\"ottsche type in abelian surfaces.
\end{abstract}

\maketitle

\tableofcontents

\section{Introduction}

This paper is the second in a series of three papers which study enumerative invariants of abelian surfaces from the tropical point of view. While the first paper focuses on the enumeration of genus $g$ curves in a fixed class passing through $g$ points, this paper is dedicated to the enumeration of genus $g$ curves in a fixed linear system that are also subject to point conditions.

\subsection{Abelian surfaces and tropical tori}

\subsubsection{Complex abelian surfaces.} Abelian surfaces are complex tori, \textit{i.e.} the quotient $\CC A=\CC^2/\Lambda$ of the complex vector space $\CC^2$ by some rank $4$ lattice $\Lambda$, that is subject to some condition, namely the existence of a positive line bundle called polarization on $\CC A$. Not every complex torus can be endowed with the choice of a polarization: this imposes conditions on a matrix spanning the lattice, known as \textit{Riemann bilinear relations}. These are proved in \cite{griffiths2014principles}. Abelian surfaces have a natural structure of additive group inherited from the vector space structure of $\CC^2$. Through the use of the exponential map, it is also possible to give a multiplicative description of an abelian surface as a quotient of the algebraic complex torus $(\CC^*)^2$ by some rank $2$ lattice. This description is more adapted to define tropical counterparts to abelian surfaces, as it provides a logarithmic map on $\CC A=(\CC^*)^2/\Lambda$.

\subsubsection{Tropical abelian surfaces.} Tropicalizing the above definition, a tropical torus is obtained as a quotient of the vector space $\RR^2$ by some rank $2$ lattice. In the rest of the paper, the vector space is $N_\RR=N\otimes\RR$, where $N$ is some rank $2$ lattice, and the lattice by which we quotient is still denoted by $\Lambda$, with an inclusion $S:\Lambda\to N_\RR$. As in the complex case, a tropical torus is a tropical abelian surface if it can be endowed with the choice of a polarization, \textit{i.e.} a positive line bundle. We have a natural analog of the Riemann bilinear relation expressing the condition on the lattice $\Lambda$ so that $\TT A=N_\RR/\Lambda$.

\subsubsection{Tropicalizing a family of complex tori.} The two previous definitions relate as it is possible to consider specific families of abelian surfaces $N_{\CC^*}/\Lambda_t$, called Mumford families, that ``tropicalize" to a tropical abelian surface $\TT A=N_\RR/\Lambda$. It is then possible to apply correspondence techniques to resolve enumerative problems inside complex abelian surfaces only by studying tropical enumerative problems.

\subsubsection{Curves in abelian surfaces.} We consider a complex abelian surface $\CC A=N_{\CC^*}/\Lambda$ endowed with a polarization $\L$. It is possible to assume up to a change of basis that the pull-back of $\L$ to $N_{\CC^*}$. Sections of $\L$ can thus be viewed as certain specific holomorphic functions on $N_{\CC^*}$ satisfying a quasi-periodic relation. These are called $\theta$-functions, and their definition is recalled in section \ref{section theta function}. The zero locus of a $\theta$-function is a curve in the linear system $|\L|$.

\subsection{Enumerative geometry and Gromov-Witten invariants}

\subsubsection{Curves and enumerative problems.} Zero loci of sections of $\L$ provide curves in the abelian surface, and it natural to try to count curves of a fixed genus that are subject to some constraints. Furthermore, using the group structure of the abelian surface, it is possible to translate curves. However, the translate of a curve does not usually belong to the same linear system anymore. It is possible to show (see \cite{bryan1999generating}) that the dimension of the deformation space of a genus $g$ curve realizing a fixed homology class $C\in H_2(\CC A,\ZZ)$ is $g$. This leads to the following two enumerative problems:
\begin{itemize}[label=$\circ$]
\item How many genus $g$ curves in the class $C$ pass through $g$ points ?
\item How many genus $g$ curves in a fixed linear system pass through $g-2$ points ?
\end{itemize}
The first paper in the series focuses on the first enumerative problem, we now deal with the second.

\medskip

As it happens, the answer to both problems does not depend on the choice of the points, polarization up to translation nor the abelian surface. They are denoted by $\N_{g,C}$ and $\N_{g,C}^{FLS}$ respectively. These invariants were already studied in \cite{bryan1999generating}. They coincide with Gromov-Witten invariants of abelian surfaces, as defined in \cite{bryan1999generating}. Their definition differs a little from usual Gromov-Witten invariants because the choice of complex structure on an abelian surface is not generic, and a generic choice would lead to the surface having no curve at all. See \cite{bryan1999generating} for more details.

\medskip

More generally, Gromov-Witten invariants are obtained by integrating cohomology classes over some virtual fundamental class in the moduli space of curves inside a specific variety, here an abelian surface. The moduli space of curves with marked points is endowed with an evaluation map to the variety. Integrating pull-back of cohomology classes Poincar\'e dual to geometric constraints by the evaluation map amounts to count curves satisfying these geometric constraints. However, there are many other cohomology classes that it is possible to integrate, such as $\lambda$-classes, $\psi$-classes, \dots The invariants from \cite{bryan1999generating} were thus generalized in \cite{bryan2018curve} to a wider families of Gromov-Witten invariants. The cohomology classes of an algebraic complex variety might be Poincar\'e dual to algebraic cycles, and thus the computation of the corresponding Gromov-Witten invariants amounts to the solving of some algebraic problem. This is the case for some toric varieties. However, this is not always the case. For abelian surfaces, there are new classes that are not dual to algebraic cycles, for instance classes in $H^{0,1}(\CC A)$ and $H^{1,0}(\CC A)$, providing other Gromov-Witten invariants.

\medskip

\subsubsection{Parametric or implicit curves.} The point of view of Gromov-Witten theory is to consider parametrized complex curves in a fixed homology class $C\in H_2(\CC X,\ZZ)$, where $\CC X$ is a complex variety. Counting curves in a fixed linear system uses more the implicit point of view on curves: curves are sections of a fixed line bundle, \textit{i.e.} they are given with an equation. Both points of view are equivalent when the Picard group of $\CC X$ is discrete in the sense that the data of the homology class $C$ uniquely determines the line bundle $\L$ and the linear system $|\L|$. However, when $\CC X=\CC A$ is for instance an abelian surface, the Picard group is not discrete as it contains a part of dimension $h^{1,0}=2$. Imposing the linear system acts as a codimension $2$ condition on the space of curves. This corresponds to the fact that translates of a curve are not section of the same line bundle anymore. In \cite{bryan1999generating}, J. Bryan and N. Leung proved that it is possible to transform condition on the linear system into conditions on the curve. For instance, being part of a fixed linear system becomes meeting four $1$-dimensional cycles whose homology classes span $H_1(\CC A,\ZZ)$.

\medskip

In this paper, we adopt a parametric point of view, so that we also need to transform the linear system condition into a more manageable data. We adopt a point of view slightly different from \cite{bryan1999generating} by instead recovering the line bundle $\O(\CCC)$ associated to a complex curve $\CCC$ in terms of integrals of certain $1$-forms on the curve. More precisely, the statement is as follows. We use the following facts. We refer to section \ref{section curves to linear system} for more details.
	\begin{itemize}[label=-]
	\item The Picard group of the abelian surface $\CC A=N_{\CC^*}/\Lambda$ is isomorphic to $\Lambda_{\CC^*}^*/M$, where $M$ is the dual lattice of $N$.
	\item For a curve $\CCC\subset N_{\CC^*}$ and a circle $\gamma$ inside $\CCC$ realizing a  homology class $n\in N$, its moment is defined as $\exp\left(\mu_\gamma=\frac{1}{2i\pi}\int_\gamma \log z\frac{\dd w}{w}\right)$, where $z$ and $w$ are the coordinates on $N_{\CC^*}$ such that $\log z$ is well-defined on $\gamma$. (it is the monomial in $M=N^*$ Poincar\'e dual to $n$)
	\item An element $\lambda\in\Lambda$ gives a $3$-dimensional variety $X_\lambda$ with boundary inside $N_{\CC^*}$. It is obtained as the preimage of a path in $N_\RR$ by the logarithmic map $N_{\CC^*}\to N_\RR$.
	\end{itemize}

The statement is as follows.

\begin{theom}\ref{theorem relation moment linear system complex}
Given two curves $\CCC_0$ and $\CCC_1$ in $\CC A$ lifted to periodic curves inside $N_{\CC^*}$. The line bundle $\O(\CCC_1-\CCC_0)$ is represented by the element of $\Lambda_{\CC^*}$ that maps $\lambda\in\Lambda$ to the product of the moments of the circles resulting from the intersection $X_\lambda\cap(\CCC_0\cup\CCC_1)$.
\end{theom}

This statement is a generalization of the $1$-dimensional statement expressing the line bundle associated to a degree $0$ divisor on an elliptic curve. See section \ref{section curves to linear system} for more details.

\subsubsection{Previous computations.} All the above invariants introduced above were computed in both \cite{bryan1999generating} and \cite{bryan2018curve} in case the class of curves $C$ is called primitive. This allows to prove some regularity statement. For instance, they show that certain generating series of these invariants are quasi-modular forms. To my knowledge, the computation for non-primitive classes remains open. This paper adresses the computation of $\N_{g,C}^{FLS}$ for non-primitive classes using the tropical method. Hopefully, generalizations of the statements from \cite{bryan2018curve} should be possible with this approach.

\medskip

We also mention the work of L. Halle and S. Rose \cite{halle2017tropical} that also count tropical curves in abelian varieties using a completely different method: they study maps between tropical tori and maps from a curve to its Jacobian. Their method leads to the computation of some of the above invariants but generalizes in higher dimension.

\subsection{Tropical geometry and correspondence theorems}

\subsubsection{Correspondence theorems in toric varieties.} Once defined, it remains to compute the invariants, in order to for instance study their regularity or their generating series. As the invariants do not depend on the choice of the constraints nor the choice of the abelian surface, it is now a classical method to try to compute them close to the \textit{tropical limit}. This approach was first implemented by G. Mikhalkin in \cite{mikhalkin2005enumerative} for computing enumerative invariants of toric surfaces. The result is known as \textit{correspondence theorems}. Since, other versions of correspondence theorems have been proved in various settings with different techniques, see for instance \cite{nishinou2006toric}, \cite{shustin2002patchworking}, \cite{shustin2004tropical}, \cite{tyomkin2017enumeration} or \cite{mandel2020descendant}.

\medskip

Close to the tropical limit, a complex curve break into several pieces whose structure is encoded in a tropical curve. Tropical curves are graphs whose edges have integer slope that satisfy the balancing condition at each of their vertices. Tropical curves were first defined as graphs in $\RR^2$, or $\RR^n$. See for instance \cite{brugalle2014bit}. It is possible to generalize their definition to any manifold whose tangent bundle contains a natural lattice. This is the case of tropical tori. Adhoc correspondence theorems can then relate the study of tropical curves inside such manifolds, called \textit{affine integer manifolds}, to enumerative or Gromov-Witten invariants of suitable complex manifolds.

\subsubsection{Correspondence in abelian surfaces.} Getting out of the toric situation, it is still possible to prove correspondence theorems by generalizing methods from the toric case. For instance, in \cite{blomme2021floor}, the author adapts methods from \cite{nishinou2006toric} to compute Gromov-Witten invariants of complex manifolds that are line bundles over an elliptic curve. However, this requires to consider families of complex varieties while correspondence for toric varieties could be seen as happening in the same variety. Concretely, in the situation of the paper, it means that we have to consider families of complex abelian surfaces $\CC A_t$ rather than a single abelian surface. 

\medskip

Concerning abelian surfaces, a correspondence theorem was first proved by T. Nishinou \cite{nishinou2020realization}. More precisely, the main result of \cite{nishinou2020realization} is twofold: it consists in a realization theorem that expresses the possible deformations of a given tropical curve, and the second that uses the description of the deformations to give an expression of the deformations that satisfy $g$ points constraints. This way, a correspondence theorem can be seen as a recipe to get a multiplicity $m_\Gamma^\CC$ out of a tropical curve, so that the count of tropical curves solution to a suitable enumerative problem with this multiplicity gives the desired invariant. In the first paper of the series, we proved a product expression for the complex multiplicity provided by Nishinou's correspondence theorem.

\medskip

\subsubsection{Correspondence for linear systems.} In the present paper, we adapt the proof of the correspondence theorem from \cite{nishinou2020realization} to work for the case of curves in a fixed linear system as well. It is possible to adapt the proof using the parametric point of view thanks to the expression of the linear system constraints provided by Theorem \ref{theorem relation moment linear system complex}.

\medskip

Let $\P_t$ be a $g-2$ point configuration in a family of abelian surfaces $\CC A_t$ that tropicalizes to $\P$ inside $\TT A$. See section \ref{section correspondence} for more details. We have the following correspondence theorem.

\begin{theom}\ref{theorem correspondence}
Let $h:\Gamma\to \TT A$ be a parametrized tropical curve passing through $\P$ and in a fixed linear system. The number of genus $g$ complex curves in the fixed linear system passing through $\P_t$ and that tropicalize to $h:\Gamma\to\TT A$ is
$$m_\Gamma^\CC=|\ker\Psi_{\CC^*}|\prod_{e\in E(\Gamma')}w_e.$$
In particular, we have that $N_{g,C}^{FLS}(\TT A,\P)$ does not depend on $\P$ and $\TT A$ as long as these choices are generic, and $N_{g,C}^{FLS}=\N_{g,C}^{FLS}$.
\end{theom}

The map $\Psi$ involved in the theorem is defined in section \ref{section correspondence}. It plays a role similar to the evaluation map $\Theta$ from \cite{nishinou2020realization}. Notice that the theorem provides a new complex multiplicity with which to count tropical curves, and it may differ from the one introduced in \cite{nishinou2020realization} and used in the first paper of this series. Its expression using $\Psi$ is not that important since we give below a concrete expression for $m_\Gamma^\CC$. The number $N_{g,C}^{FLS}$ is the number of genus $g$ tropical curves in a fixed linear system passing through a generic configuration of $g-2$ points.

\medskip

In fact, we also have a product expression provided by the following theorem. Let $h:\Gamma\to\TT A$ be a tropical curve passing through a generic point configuration $\P$ of $g-2$ points and that belongs to a fixed linear system. The complement of $\P$ inside $\Gamma$ is connected and has genus $2$. It retracts onto a genus $2$ subgraph $\Sigma\subset\Gamma$. Let $\Lambda_\Gamma^\Sigma$ be the index of $H_1(\Sigma,\ZZ)$ inside $H_1(\TT A,\ZZ)\simeq\Lambda$. Let also $\delta_\Gamma$ be the gcd of the weights of the edges of the curve.

\begin{theom}\ref{theorem multiplicity formula}.
One has $m_\Gamma^\CC=\delta_\Gamma\Lambda_\Gamma^\Sigma m_\Gamma$, where $m_\Gamma=\prod_V m_V$ is the usual multiplicity, and $m_V$ is the usual vertex multiplicity.
\end{theom}

The new complex multiplicity possesses a new term $\Lambda_\Gamma^\Sigma$. Tropically, its presence can be explained by the appearance of new kind of walls. These new walls are as follows. Usually, when moving the constraints, one of the following events, called \textit{wall}, can happen:
\begin{itemize}[label=$\circ$]
\item An edge is contracted leading to the appearance of a quadrivalent vertex. In this case, two solutions may become a unique solution on the other side of the wall.
\item A cycle is contracted to a segment, leading to a pair of quadrivalent vertices linked by a pair of parallel edges. In this case, one curve is replaced by another.
\item Last, one of the marked points meets a vertex of the curve.
\end{itemize}
For the latter kind of wall, in the usual case of toric surfaces, the fact that the complement of the marked point in the curve is a forest allows one to prove that the marked point can only go from on edge adjacent to the vertex to another, so that there are only two out of the three adjacent combinatorial types that can provide a solution. Here, if a marked point merges with a vertex of $\Sigma$, the marked point may move onto any of the adjacent edges, and we can have two solutions becoming one on the other side of the wall. The local invariance is ensured by the new term $\Lambda_\Gamma^\Sigma$.



\subsection{Refined invariants for curves in linear systems}

The usual complex multiplicity provided by the correspondence theorem from \cite{mikhalkin2005enumerative} expresses as a product over the vertices of the tropical curve. In \cite{block2016refined}, F. Block and L. G\"ottsche proposed to refine this multiplicity into a Laurent polynomial one by replacing the vertex multiplicities by their quantum analog. In the case of toric surfaces, I. Itenberg and G. Mikhalkin proved in \cite{itenberg2013block} that the count of tropical curves with fixed genus and degree passing through the right number of points with refined multiplicity is invariant. These results were generalized in various settings, see for instance \cite{schroeter2018refined}, \cite{blechman2019refined}, \cite{gottsche2019refined}, \cite{blomme2021floor}, \cite{blomme2021refinedtrop}.

\medskip

In our case, despite the appearance of the new term $\Lambda_\Gamma^\Sigma$, the complex multiplicity provided by Theorem \ref{theorem correspondence} and Theorem \ref{theorem multiplicity formula} still expresses as a product over the vertices. Thus, we can define the refined multiplicity $m_\Gamma^q=\prod_V \frac{q^{m_V/2}-q^{-m_V/2}}{q^{1/2}-q^{-1/2}}$. We then introduce the following enumerative counts:
\begin{align*}
N_{g,C,k}^{FLS}(\TT A,\P) & = \sum_{\substack{h(\Gamma)\supset\P \\ \delta(\Gamma)=k}} \Lambda_\Gamma^\Sigma m_\Gamma \in\NN , \\
BG_{g,C,k}^{FLS}(\TT A,\P) & = \sum_{\substack{h(\Gamma)\supset\P \\ \delta(\Gamma)=k}} \Lambda_\Gamma^\Sigma  m^q_\Gamma \in\ZZ[q^{\pm 1/2}] .\\
\end{align*}
counting curves with a fixed gcd, and the following counts for curves without a gcd condition:
\begin{align*}
M_{g,C}^{FLS}(\TT A,\P) &  = \sum_{h(\Gamma)\supset\P} \Lambda_\Gamma^\Sigma  m_\Gamma = \sum_{k|\delta(C)} N_{g,C,k}^{FLS}(\TT A,\P) \in\NN , \\
N_{g,C}^{FLS}(\TT A,\P) &  = \sum_{h(\Gamma)\supset\P} \delta_\Gamma \Lambda_\Gamma^\Sigma  m_\Gamma = \sum_{k|\delta(C)}k N_{g,C,k}^{FLS}(\TT A,\P) \in\NN , \\
R_{g,C}^{FLS}(\TT A,\P) & = \sum_{h(\Gamma)\supset\P} \delta_\Gamma \Lambda_\Gamma^\Sigma  m^q_\Gamma = \sum_{k|\delta(C)}k BG_{g,C,k}^{FLS}(\TT A,\P) \in \ZZ[q^{\pm 1/2}] , \\
BG_{g,C}^{FLS}(\TT A,\P) & = \sum_{h(\Gamma)\supset\P} \Lambda_\Gamma^\Sigma  m^q_\Gamma = \sum_{k|\delta(C)} BG_{g,C,k}^{FLS}(\TT A,\P) \in \ZZ[q^{\pm 1/2}] . \\
\end{align*}
We already know that the count $N_{g,C}^{FLS}(\TT A,\P)$ does not depend on the choice of $\P$ as long as it is generic, nor the choice of $\TT A$ as long as it is generic. This invariance is provided by the correspondence theorem. The invariance for other counts cannot be deduced from any complex invariants, and has thus to be studied on its own, looking at the \textit{walls} depicted above.

\begin{theom}\ref{theorem point invariance}
The refined count $BG_{g,C,k}^{FLS}(\TT A,\P)$ of genus $g$ curves in the class $C$ with fixed gcd passing through $\P$ in the fixed linear system (and thus all the others) does not depend on the choice of $\P$ and the line bundle as long as it is generic.
\end{theom}

We remove $\P$ from the notation to denote the associated invariant. Then, we have an invariance statement regarding the choice of the abelian surface $\TT A$.

\begin{theom}\ref{theorem surface invariance}
The refined invariant $BG_{g,C,k}^{FLS}(\TT A)$ (and thus all the others) does not depend on the choice of $\TT A$ as long as it is chosen generically among the surfaces that contain curves in the class $C$.
\end{theom}

Interpretations of refined invariants in toric varieties have already been proved: they correspond both to refined signed count of real curves according to the value of some quantum index, see \cite{mikhalkin2017quantum}, \cite{blomme2021refinedreal}, or to generating series of Gromov-Witten invariants with insertions of $\lambda$-classes, as proven by P. Bousseau in \cite{bousseau2019tropical}. It is also conjectured that they correspond to the refinement of the Euler characteristic of some relative Hilbert scheme by the Hirzebruch genus, see \cite{gottsche2014refined}. Although it should be possible to adapt some methods of the above citations, the meaning of the refined invariants in abelian surfaces remains open. Computations from \cite{bryan2018curve} and results from \cite{bousseau2019tropical} suggest that refined invariants enable the computation of Gromov-Witten invariants with insertions of $\lambda$-classes in this situation as well.

\medskip

The refinement of the complex multiplicity by the refined multiplicity proposed in \cite{block2016refined} adapts in our situation due to the appearance of the product of vertex multiplicities. The term $\delta_\Gamma$ could be replaced by anything else since we have invariance for the count of tropical curves with a fixed gcd. It would be interesting to find a way to refine the new term $\Lambda_\Gamma^\Sigma$.

\subsection{Plan of the paper}

The paper is organized as follows. The second section deals with line bundles on abelian surfaces, their sections called $\theta$-functions, and how to relate the parametric and implicit points of view for curves both in the complex and tropical case. The third section studies the enumerative problem considered in this paper. It states the correspondence theorem, and results concerning tropical multiplicities and the invariants. The fourth section is devoted to the proof of the results stated in the third. The last section provides small examples. We refer to the last paper for more examples of computations using the pearl diagram algorithm.

\textit{Acknowledgments.} Research is supported in part by the SNSF grant 204125.

\section{Line bundles and curves in linear systems}

In the following section, we consider a tropical torus $\TT A=N_\RR/\Lambda$, where $N$ and $\Lambda$ are two lattices of rank $2$, and $N_\RR$ is the real vector space $N\otimes\RR$. The matrix of the inclusion $\Lambda\hookrightarrow N_\RR$ is denoted by $S$.

\subsection{Line bundles on abelian surfaces}

Line bundles on tropical tori are defined analogously to the classical case. See \cite{griffiths2014principles} for the classical setting, and \cite{halle2017tropical} or \cite{mikhalkin2008tropical} for a more complete reference in the tropical setting. In the classical case, the group of isomorphism classes of line bundles on a variety $Y$ is the cohomology group $\mathrm{Pic}\ Y= H^1(Y,\O^\times_Y)$. In the tropical world, the sheaf of affine functions with integer slope $\mathrm{Aff}_\ZZ$ plays that role. Locally, an affine function is of the form $x\in U\subset N_\RR\mapsto \langle m,x\rangle +l\in\RR$.

\begin{defi}
A isomorphism class of line bundles on $\TT A$ is an element of the \textit{Picard group} $H^1(\TT A,\mathrm{Aff}_\ZZ)$, where $\mathrm{Aff}_\ZZ$ is the sheaf of affine functions with integer slope, \textit{i.e.} slope in $M$.
\end{defi}

In tropical geometry, the following exact sequence plays the role of the exponential sequence:
$$0\rightarrow\RR\rightarrow\mathrm{Aff}_\ZZ\rightarrow M\rightarrow 0.$$
The second arrow maps a function to its slope $m\in M$, which is a section of the cotangent bundle. We then have the long exact sequence associated to the short exact sequence:
$$\cdots \to H^0(\TT A,M) \to H^1(\TT A,\RR) \to H^1(\TT A,\mathrm{Aff}_\ZZ) \xrightarrow{c_1} H^1(\TT A,M) \to H^2(\TT A,\RR)\to\cdots.$$
Morever, we have the following isomorphisms:
\begin{itemize}[label=-]
\item The group $H^0(\TT A,M)$ of global $1$-forms with integer slope is just $M$,
\item As $H_1(\TT A,\ZZ)\simeq \Lambda$, we have $H^1(\TT A,\RR)\simeq \Lambda^*_\RR$,
\item We similarly have $H^1(\TT A,M)\simeq \Lambda^*\otimes M=\hom(\Lambda,M)$.
\end{itemize}

On the left side of the sequence, the cobordism map $M\to \Lambda^*_\RR$ is just the dual map to the inclusion $S:\Lambda\to N_\RR$. By assumption, it is an injection. Thus, we have an injection of the dual torus into the Picard group:
$$\TT A^\vee=H^1(\TT A,\RR)/H^0(\TT A,M)=\Lambda^*_\RR/M\hookrightarrow H^1(\TT A,\mathrm{Aff}_\ZZ).$$

Meanwhile, on the right side of the sequence, the map $c_1:H^1(\TT A,\mathrm{Aff}_\ZZ)\to H^1(\TT A,M)\simeq \hom(\Lambda,M)$ is called the Chern class map. Let $\L$ be a line bundle and $c=c_1(\L)$ denotes its Chern class. By definition, we have $c:\Lambda\to M$. Then, using the inclusion $S:\Lambda\hookrightarrow N_\RR$, we set the following bilinear form:
$$B(\lambda_1,\lambda_2)=c(\lambda_1)(S\lambda_2),$$
that we also write $c(\lambda_1)(\lambda_2)$ or $\langle c(\lambda_1),\lambda_2\rangle$ if the inclusion $S:\Lambda\to N_\RR$ is implicit. We have the following proposition, already noticed in \cite{mikhalkin2008tropical}.

\begin{prop}
For any $\L\in H^1(\TT A,\mathrm{Aff}_\ZZ)$, the bilinear form $B$ induced by its Chern class $c:\Lambda\to M$ is symmetric.
\end{prop}

\begin{proof}
The map to $H^2(\TT A,\RR)=\Lambda^2 H^1(\TT A,\RR)$ corresponds to the skew-symmetrization. Hence the symmetry. 
\end{proof}

These considerations allow us to find a decomposition of the Picard group $H^1(\TT A,\mathrm{Aff}_\ZZ)$: it contains the dual torus $\TT A^\vee$, and surjects to a discrete lattice: the elements of $\hom(\Lambda,M)$ that induce a symmetric pairing on $\Lambda_\RR$. Let us describe this lattice in coordinates. Choosing basis of $\Lambda$ and $M$, we have an identification of $\hom(\Lambda,M)$ with $\M_2(\ZZ)$, and let $S$ be the matrix of the inclusion $\Lambda\hookrightarrow N_\RR$.

\begin{prop}
A matrix $c$ belongs to $\mathrm{Im}\ c_1$ if and only if $S^T c$ is symmetric: $S^T c\in\S_2(\RR)$.
\end{prop}

\begin{rem}
Notice that this condition is equivalent to $Sc^{-1}$ being symmetric when $c$ is invertible.
\end{rem}

\begin{proof}
The dual map $M\to\Lambda^*_\RR$ is given by the matrix $S^T$. Meanwhile, $c$ is the matrix of a map $\Lambda\to M$. Hence, the matrix of the bilinear map $B$ is given by $S^T c$, which is by composition a map $\Lambda\to\Lambda_\RR^*$.
\end{proof}

Thus, the image of the Chern class map is the intersection of the lattice $\hom(\Lambda,M)\subset\hom(\Lambda,M)_\RR$ with some hyperplane depending on the inclusion $S$. In particular, if it is chosen generically, it is empty. One other way to view it is that a given class in $\hom(\Lambda,M)$ being realized as a Chern class imposes a condition on the inclusion $S$.

\begin{expl}
Assume we have $S=\begin{pmatrix}
\alpha & \beta \\
\gamma & \delta \\
\end{pmatrix}$ and $c=\begin{pmatrix}
n & 0 \\
0 & 1 \\
\end{pmatrix}$. Then $c$ is realized as a Chern class in $\TT A$ if and only if $n\beta=\gamma$.
\end{expl}

We finish by giving the definition of a polarization.

\begin{defi}
The choice of a \textit{polarization} on $\TT A$ is the data of a $c$ such that $B$ is symmetric and positive definite bilinear form on $\Lambda$.
\end{defi}

Concretely, a line bundle with chern class $c$ can be seen as the quotient of the trivial line bundle $N_\RR\times\RR$ by the following action of $\Lambda$, extending the action by translation on $N_\RR$:
$$\lambda\cdot(x,\xi)=(x+\lambda,\xi+c(\lambda)(x)+\beta_\lambda).$$
The $\beta_\lambda$ are some real numbers. For this formula to define an action, we need to have the following relation:
$$\beta_{\lambda+\mu}=\beta_\lambda+\beta_\mu+c(\lambda)(\mu).$$
This is another way to see that $c$ defines a symmetric form. It implies that $\beta_\lambda=\frac{1}{2}c(\lambda)(\lambda)$ up to a linear map in $\lambda$. This linear map corresponds to the $\Lambda^*_\RR/M$ part in the Picard group and is defined by translations of the line bundle.

\subsection{complex and tropical $\theta$-functions}
\label{section theta function}

The goal of this section is to define tropical $\theta$-function in a setting more general than the one of \cite{mikhalkin2008tropical}, recall the definition of $\theta$-functions in the complex setting, as in \cite{griffiths2014principles} but with multiplicative notations, and relate both when considering a Mumford family of tori $\CC A_t$.

\subsubsection{Tropical $\theta$-functions.} Let us consider the line bundle $\L$ with fixed Chern class $c$, that is the quotient of $N_\RR\times\RR$ by the action $\lambda\cdot(x,\xi)=(x+\lambda,\xi+c(\lambda)(x)+\beta_\lambda)$ for $\beta_\lambda=\frac{1}{2}c(\lambda)(\lambda)$. All line bundles with Chern class $c$ are translate of this specific line bundle.

\begin{defi}
A $\theta$-function for $\L$ is a global convex piecewise affine function $\Theta$ on $N_\RR$ with integer slope and satisfying the quasi-periodicity condition
$$\Theta(x+\lambda)=\Theta(x)+c(\lambda)(x)+\beta_\lambda.$$
\end{defi}

We have the following proposition that expresses all the $\theta$-functions.

\begin{prop}
Let $\{m_0\}$ be a system of representatives of $M$ modulo $c(\Lambda)$. Any $\theta$-function is of the form $\theta(x)=\max_{m_0}\left( a_{m_0}+\theta_{m_0}(x)\right)$ for some choice of coefficients $a_{m_0}$, where
$$\Theta_{m_0}(x)=\langle m_0,x\rangle + \max_{l\in\Lambda}\left( \langle c(l),x\rangle-\frac{1}{2}\langle c(l),l\rangle -\langle m_0,l\rangle  \right),$$
\end{prop}

\begin{proof}
As a convex function on $N_\RR$, a $\theta$-function $\Theta$ is determined by its Legendre transform:
$$\widehat{\Theta}(m)=\max_{x\in N_\RR}\left( \langle m,x\rangle -\Theta(x) \right).$$
The quasi-periodicity condition on $\Theta$ translates to the following quasi-periodicity condition on $\widehat{\Theta}$:
\begin{align*}
\widehat{\Theta}(m+c(\lambda)) & = \max_{x\in N_\RR}\left( \langle m+c(\lambda),x\rangle -\Theta(x) \right) \\
 & = \max_{x\in N_\RR}\left( \langle m,x\rangle +\langle c(\lambda),x\rangle -\Theta(x) - \beta_{-\lambda} \right) +\beta_{-\lambda} \\
 & = \max_{x\in N_\RR}\left( \langle m,x\rangle -\Theta(x-\lambda)\right)  +\beta_{-\lambda}\\
 & = \widehat{\Theta}(m)-\langle m,\lambda\rangle +\beta_{-\lambda} .\\
\end{align*}
It means that the values of the function $\widehat{\Theta}$ are determined by a system of representatives modulo $c(\Lambda)\subset M$. We check that the functions $\Theta_{m_0}$ satisfy the relations, as in \cite{mikhalkin2008tropical}. They are well-defined since the function $l\mapsto\langle c(l),l\rangle$ is positive definite. The result follows.
\end{proof}

\begin{rem}
The tropical $\theta$-functions are functions on $N_\RR$, but they can be seen as section of line bundles on $\TT A=N_\RR/\Lambda$.
\end{rem}

\subsubsection{Complex $\theta$-functions.} Concerning the complex line bundles on a complex abelian surface, we refer to the corresponding section of \cite{griffiths2014principles} for a detailed version of the following. Consider the abelian surface $\CC A=N_{\CC^*}/\Lambda$. This is a multiplicative presentation of an abelian surface. In  \cite{griffiths2014principles}, an abelian surface is presented as a quotient $N_\CC/\Omega$ for a matrix $\Omega=\begin{pmatrix}
I_2 & \tilde{Z}
\end{pmatrix}$, where $\tilde{Z}:\Lambda\to N_\CC$ is some matrix. The relation between both notations is provided by the exponential map
$$\tilde{z}\in N_\CC\longmapsto z=e^{2i\pi \tilde{z}}\in N_{\CC^*}.$$
In order to relate the complex setting to the tropical setting, we need to adopt the multiplicative notation. However, the reader might be more accustomed to the additive notation. Thus, we try to specify both notations as often as possible. Notation with a $\sim$ are for the additive setting, and without for the multiplicative one.

\medskip

Consider a line bundle on $\CC A$ whose pull-back to $N_{\CC^*}$ is the trivial bundle $N_{\CC^*}\times\CC$. Its Chern class is a skew-symmetric map on the lattice $N\oplus\Lambda$. Assume it is of the form $Q=\begin{pmatrix}
0 & -c \\
c^T & 0 \\
\end{pmatrix}$, where $c:\Lambda\to M$, and satisfies the Riemann bilinear relation from \cite{griffiths2014principles}: $\tilde{Z}c^{-1}$ is symmetric and $\im(-\tilde{Z}c^{-1})$ is positive. Moreover, there exists a function $\alpha:\Lambda\to\CC^*$ such that the line bundle is obtained quotienting $N_{\CC^*}\times\CC$ by the action of $\Lambda$:
$$\lambda\cdot(z,\xi)=(\lambda\cdot z, \alpha_\lambda z^{c(\lambda)}\xi).$$
For this relation to define an action, we need $\alpha$ to satisfy the relation $\alpha_{\lambda+\mu}=\alpha_\lambda\alpha_\mu \lambda^{c(\mu)}$. In particular, $\lambda^{c(\mu)}=\mu^{c(\lambda)}$, which amounts to the Riemann bilinear relation $\tilde{Z} c^{-1}$ is symmetric. For what follows, we fix $\alpha_\lambda=e^{i\pi \langle c(\lambda),Z\lambda\rangle}$, which satisfies the relation. This fixes the line bundle which was previously defined up to translation.

\begin{defi}
A $\theta$-function is a holomorphic function on $N_\CC$ (resp. $N_{\CC^*}$) that satisfies the following relations:
$$\left\{ \begin{array}{l}
\theta(\tilde{z}+n) =\theta(\tilde{z}) \text{ for }n\in N\subset N_\CC,\\
\theta(\tilde{z}+\tilde{Z}\lambda) =\alpha_\lambda e^{2i\pi \langle c(\lambda),\tilde{z}\rangle}\theta(\tilde{z}) \\
\end{array} \right.
\text{ or multiplicatively } \theta(\lambda\cdot z)=\alpha_\lambda z^{c(\lambda)} \theta(z).$$
\end{defi}

These relations need only to check it on a basis of the lattice. If $(e_1,e_2)$ is a basis of $N$ and $(\lambda_1,\lambda_2)$ a basis of $\Lambda$,
$$\left\{ \begin{array}{l}
\theta(\tilde{z}+e_j) =\theta(\tilde{z}) \\
\theta(\tilde{z}+\tilde{Z}\lambda_j) =e^{2i\pi \langle c(\lambda_j),\tilde{z}\rangle +i\pi \langle c(\lambda_j),\tilde{Z}\lambda_j\rangle}\theta(\tilde{z}) \\
\end{array} \right.
\text{ or multiplicatively } \theta(\lambda_j\cdot z)=e^{i\pi\langle c(\lambda_j),\tilde{Z}\lambda_j\rangle}z^{c(\lambda_j)} \theta(z).$$

\begin{prop}
Any $\theta$-function is a linear combination of the following $\theta$-functions:
\begin{align*}
 & \theta_{m_0}(\tilde{z})=e^{2i\pi\langle m_0,\tilde{z}\rangle}\sum_{l\in\Lambda}e^{-i\pi\langle c(l),Zl\rangle -2i\pi\langle m_0,Zl\rangle}e^{2i\pi\langle c(l),\tilde{z}\rangle} \\
\text{ or multiplicatively } & \theta_{m_0}(z)=z^{m_0}\sum_{l\in\Lambda}e^{-i\pi\langle c(l),Zl\rangle -2i\pi\langle m_0,Zl\rangle}z^{c(l)}. \\
\end{align*}
\end{prop}

\begin{proof}
Following \cite{griffiths2014principles}, any $\theta$-function can be developped in power series as follows
$$\theta(\tilde{z})=\sum_{m\in M}a_m e^{2i\pi\langle m,\tilde{z}\rangle}
\text{ or multiplicatively }
\theta(z)=\sum_{m\in M}a_m z^m.$$
As in the tropical case, the relation imposes the following relations on the coefficients
$$a_{m-c(\lambda)}=e^{2i\pi\langle m,\tilde{Z}\lambda\rangle}e^{-i\pi\langle c(\lambda),\tilde{Z}\lambda\rangle}a_m.$$
Thus, all the coefficients are determined by the coefficients $a_{m_0}$ for $\{ m_0\}$ a system of representatives of the quotient $M/c(\Lambda)$. We check that the functions $\theta_{m_0}$ satisfy the relations. These are well-defined since $l\mapsto \langle c(l),Zl\rangle$ has negative imaginary part, so that the series converge. The result follows.
\end{proof}

\subsubsection{Varying the surface.} We now consider a family of abelian surfaces $\CC A_t=N_{\CC^*}/\langle e^{A} t^S\rangle$ with a line bundle having Chern class $c:\Lambda\to M$. This corresponds to the choice $Z_t=\frac{1}{2i\pi}(A+S\log t)$, where $A$ is a complex matrix, and $S$ is an integer matrix chosen such that $Sc^{-1}$ is symmetric and positive definite. The study of the complex setting tells us that we can take as a basis of sections the $\theta$-functions having the following multiplicative expressions:
$$\theta_{m_0}(z)=z^{m_0}\sum_{l\in\Lambda} e^{-\frac{1}{2}\langle c(l),Al\rangle -\langle m_0,Al\rangle}t^{-\frac{1}{2}\langle c(l),Sl\rangle -\langle m_0,Sl\rangle}z^{c(l)}.$$
Taking $\log_t$ and the limit as $t$ goes to $\infty$, we have the following propoition.

\begin{prop}
The tropical limit of the $\theta$-function $\theta_{m_0}$ is the tropical $\theta$-function $\Theta_{m_0}$.
\end{prop}

\subsection{Curves in abelian surfaces}

\subsubsection{Parametrized curves.} For sake of completeness of the present paper, we include a quick reminder of tropical curves in abelian surfaces from the parametric point of view and refer the reader to the first paper for more details.

\begin{defi}
An \textit{abstract tropical curve} is a finite metric graph $\Gamma$. A parametrized tropical curve is a map $h:\Gamma\to\TT A$ such that:
\begin{itemize}[label=-]
\item If $e$ is an edge, $h$ is affine with integer slope on the edges of $\Gamma$. The slope is of the form $w_e u_e$, where $w_e$ is an integer called weight of the edge, and $u_e\in N$ is a primitive vector. We set $\delta_\Gamma=\mathrm{gcd}_{e}w_e$ to be the gcd of the weight of the edges.
\item For each vertex $V$ of $\Gamma$, one has the balancing condition: $\sum_{e\ni V} w_e u_e=0$, where $w_e u_e$ denotes the slope of $h$ on $e$ oriented outside $V$.
\end{itemize}
\end{defi}

\begin{defi}
Let $h:\Gamma\to\TT A$ be a parametrized tropical curve.
\begin{itemize}[label=-]
\item The curve is \textit{irreducible} if $\Gamma$ is connected.
\item If the curve is irreducible, its \textit{genus} is the first Betti number $b_1(\Gamma)$.
\item Its \textit{degree} is the class that it realizes inside $H_{1,1}(\TT A)\simeq \Lambda\otimes N$.
\item The curve is said to be \textit{simple} if $\Gamma$ is trivalent en $h$ is an immersion.
\end{itemize}
\end{defi}

We also have the following statement proved in the first paper that enables a concrete computation of the degree $C$ of a tropical curve $h:\Gamma\to \TT A$.

\begin{prop}
Let $C\in\Lambda\otimes N$ be a class and $S:\Lambda\hookrightarrow N_\RR$ be the inclusion defining the abelian surface $\TT A=N_\RR/\Lambda$. Then,
\begin{itemize}[label=-]
\item The class $C$ is realized by a tropical curve if and only if the product $CS^T\in\S_2^{++}(\R)$.
\item Given a parametrized tropical curve $h:\Gamma\to N_\RR$, the class $C\in\Lambda\otimes N$ that it realizes is obtained by adding the slopes of the edges intersected by two loops realizing a basis of $H_1(\TT A,\ZZ)$.
\end{itemize}
\end{prop}

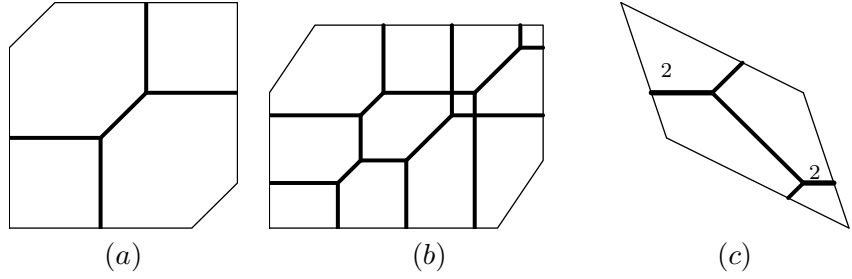
\begin{figure}
\begin{center}
\begin{tabular}{ccc}
\begin{tikzpicture}[line cap=round,line join=round,>=triangle 45,x=0.3cm,y=0.3cm]
\clip(0,0) rectangle (10,10);
\draw [line width=0.5pt] (0,0)--++ (8,0)--++ (2,2)--++ (0,8)--++ (-8,0)--++ (-2,-2)--++ (0,-8);

\draw [line width=1.5pt] (4,0)-- (4,4);
\draw [line width=1.5pt] (0,4)-- (4,4);
\draw [line width=1.5pt] (4,4)-- (6,6);
\draw [line width=1.5pt] (6,6)-- (6,10);
\draw [line width=1.5pt] (6,6)-- (10,6);

\begin{scriptsize}

\end{scriptsize}
\end{tikzpicture}
&
\begin{tikzpicture}[line cap=round,line join=round,>=triangle 45,x=0.3cm,y=0.3cm]
\clip(0,0) rectangle (14,11);
\draw [line width=0.5pt] (0,0)--++ (10,0)--++ (2,3)--++ (0,6)--++ (-10,0)--++ (-2,-3)--++ (0,-6);

\draw [line width=1.5pt] (0,2)--++ (3,0)--++ (1,1)--++ (2,0)--++ (2,2)--++ (4,0);
\draw [line width=1.5pt] (0,5)--++ (4,0)--++ (1,1)--++ (4,0)--++ (2,2)--++ (1,0);
\draw [line width=1.5pt] (3,0)--++ (0,2);
\draw [line width=1.5pt] (6,0)--++ (0,3);
\draw [line width=1.5pt] (9,0)--++ (0,6);
\draw [line width=1.5pt] (4,3)--++ (0,2);
\draw [line width=1.5pt] (8,5)--++ (0,4);
\draw [line width=1.5pt] (5,6)--++ (0,3);
\draw [line width=1.5pt] (11,8)--++ (0,1);

\begin{scriptsize}

\end{scriptsize}
\end{tikzpicture}
&
\begin{tikzpicture}[line cap=round,line join=round,>=triangle 45,x=0.3cm,y=0.3cm]
\clip(0,0) rectangle (10,10);
\draw [line width=0.5pt] (2,4)-- (0,10);
\draw [line width=0.5pt] (2,4)-- (10,0);
\draw [line width=0.5pt] (0,10)-- (8,6);
\draw [line width=0.5pt] (10,0)-- (8,6);

\draw [line width=2pt] (1.33333333333,6)-- (4,6);
\draw [line width=1.5pt] (4,6)-- (8,2);
\draw [line width=1.5pt] (4,6)-- (5.333333333,7.33333333);
\draw [line width=1.5pt] (8,2)-- (7.333333333,1.33333333);
\draw [line width=2pt] (8,2)-- (9.3333333333,2);

\begin{scriptsize}
\draw (2,7) node {$2$};
\draw (8.5,2.5) node {$2$};
\end{scriptsize}
\end{tikzpicture}
\\
$(a)$ & $(b)$ & $(c)$\\
\end{tabular}

\caption{\label{figure example tropical curves}Three examples of tropical curves in tropical tori. }
\end{center}
\end{figure}

\begin{expl}
On Figure \ref{figure example tropical curves}, we can see three different examples of tropical curves in different tropical tori $\TT A$. Each tropical torus is represented either by a parallelogram or by an hexagon whose pairs of opposite edges have been glued together. The curve $(a)$ has genus $2$ and degree $\begin{pmatrix}
1 & 0 \\
0 & 1 \\
\end{pmatrix}$, curve $(b)$ genus $5$ and degree $\begin{pmatrix}
2 & 0 \\
0 & 3 \\
\end{pmatrix}$, and curve $(c)$ genus $2$ and degree $\begin{pmatrix}
2 & 1 \\
0 & 1 \\
\end{pmatrix}$.
\end{expl}

\subsubsection{Implicit curves.} In the planar setting, tropical curve inside $N_\RR$ have two possible descriptions: they are either the image of a parametrized tropical curve, or the corner locus of a \textit{tropical polynomial}. For an abelian surface with a line bundle, the sections play the role of the tropical polynomials, and the distinguished tropical $\theta$-functions $\Theta_{m_0}$ defined in the previous section play the role of the monomials in $N_\RR$.

\begin{defi}
Let $c$ be the Chern class of a line bundle $\L$. Let $\Theta(x)=\max_{m_0}\left( a_{m_0}+\Theta_{m_0}(x)\right)$ be a tropical $\theta$-function, defining a section of $\L$. Then the corner locus of $\theta$ is called a \textit{planar tropical curve} in $\TT A$.
\end{defi}

As for tropical curves inside $N_\RR$, any planar tropical curve in $\TT A$ admits a parametrization $h:\Gamma\to\TT A$. A planar tropical curve is \textit{irreducible} if it cannot be written as the union of two planar tropical curves. The \textit{genus} of an irreducible planar tropical curve is the minimal genus among the possible parametrizations of the curve.

The following proposition relates the parametric point of view and the implicit point of view by giving the relation between the class $C$ realized by a tropical curve $\Gamma$, and the Chern class $c$ of the line bundle $O(\Gamma)$.
\begin{prop}
Let $h:\Gamma\to\TT A$ be a parametrization of a tropical curve that is a section of a line bundle $\L$ with Chern class $c:\Lambda\to M$. Let $C:\Lambda^*\to N$ be the degree of the parametrized tropical curve. Then $C$ and $c$ are Poincar\'e dual to each other: $c$ is the comatrix of $C$.
\end{prop}

\begin{proof}
The statement comes from Poincar\'e duality inside $\TT A$: the integration of $c$ over a cycle is obtained by intersecting the cycle with $C$. 
\end{proof}

Up to a multiplicative constant, it means that $C=(c^{-1})^T$. We can then check that the conditions $Sc^{-1}\in\S^{++}_2(\RR)$ and $CS^T\in\S^{++}_2(\RR)$ are equivalent. By abuse of notation, by tropical curve we mean a parametrized tropical curve or a planar tropical curve.

\begin{rem}
Notice that as in the case of tropical curves inside $N_\RR$, tropical curves are dual to some particular subdivisions of the torus $M_\RR/c(\Lambda)$.
\end{rem}

\begin{expl}
The curve on Figure \ref{figure example tropical curves} $(a)$ is the corner locus of the unique $\theta$-function associated to the principal polarization of $\TT A$. The curve on $(c)$ is the corner locus of one pf the two $\theta$-functions given by the polarization. To get the curve on $(b)$, one needs several $\theta$-functions.
\end{expl}

\subsection{From curves to linear systems}
\label{section curves to linear system}

We have seen that there are two points of views for curves inside abelian surfaces: either as parametrized curve, or as zero-locus of a section of some line bundle $\L$. Tautologically, a curve $\CCC$ is always the zero-locus of a section of the line bundle $\O(\CCC)$. The goal of this section is to relate the two points of views: how to recover the line bundle $\O(\CCC)$ from a parametrization of the curve. By that, we mean that given two curves $\CCC_0$ and $\CCC_1$ in the same homology class, how to recover the element $\O(\CCC_1-\CCC_0)$ of the dual torus $\Lambda^*_{\CC^*}/M$.

\medskip

Let $\CC A$ be an abelian surface and $\CCC_0$ and $\CCC_1$ be two curves inside $\CC A$ realizing the same homology class. The divisor $\CCC_1-\CCC_0$ defines a line bundle $\O(\CCC_1-\CCC_0)$ of degree $0$. It belongs to the zero-component of the Picard group of $\CC A$. If $\CCC_0$ and $\CCC_1$ were given as zero locus of sections of their respective line bundles represented by two $\theta$-functions $\theta_0,\theta_1:N_{\CC^*}\to\CC$, then the quotient satisfies the relation
$$\frac{\theta_1}{\theta_0}(\lambda\cdot z)=K_\lambda\frac{\theta_1}{\theta_0}(z),$$
for some $K_\lambda$ where $\lambda\in\Lambda\mapsto K_\lambda\in\CC^*$ is a morphism. As $\theta_1$ and $\theta_0$ are defined up to monomials, $K$ is only well-defined up to the action of $M$, which is not surprising since it is an element in the dual torus. If $K_\lambda=1$, it would mean that $\CCC_0$ and $\CCC_1$ belong to the same linear system: they are equivalent since $\frac{\theta_1}{\theta_0}$ descends to a meromorphic function on $\CC A$. Finding the line bundle $\O(\CCC_1-\CCC_0)$ amounts to find $K$. This is more or less straightforward if $\theta_1$ and $\theta_0$ are given. However, this is not the case if the curves are given by parametrizations $\varphi_0:\CC C_0\to\CC A$ and $\varphi_1:\CC C_1\to\CC A$.

\medskip

First, we explain how to recover $K$ from the curves in the tropical case and close to the tropical limit before getting a general statement for complex tori.

	\subsubsection{Tropically and close to the tropical limit}

Let $h_0:\Gamma_0\to\TT A$ and $h_1:\Gamma_1\to\TT A$ be two parametrized tropical curves realizing the same homology class $C$ in a tropical abelian variety $\TT A=N_\RR/\Lambda$. Both curves lifts to periodic curves inside $N_\RR$ that we denote by the same letters, and that are the corner locus of a tropical $\theta$-function $\Theta_0$ and $\Theta_1$, defined up to addition of an affine function with slope in $M$. The function $f(x)=\Theta_1(x)-\Theta_0(x)$ satisfies the quasi-periodicity relation
$$f(x+\lambda)=f(x)+K_\lambda,$$
where $\lambda\mapsto K_\lambda$ is linear. Adding a monomial to one of the $\theta$ functions changes $f$ and thus $K$. Thus, only the class of $K$ in $\Lambda^*_\RR/M$ is well-defined. We intend to recover the element $K\in\Lambda^*_\RR/M$.

\medskip

For the statement as well as the proof, recall the notion of \textit{moment of an edge} for a tropical curve. Let $h:\Gamma\to N_\RR$ be a tropical curve, $e$ be an oriented edge where $h$ has slope $w_e u_e$. The moment of the oriented edge is $\det(w_eu_e,p)$, where $p$ is any point on the edge. It corresponds to the position of the edge along a transversal axis.

	\begin{theo}\label{theorem relation moment linear system tropical}
	Let $h_0:\Gamma_0\to\TT A$ and $h_1:\Gamma_1\to\TT A$ be two parametrized tropical curves realizing the same homology class $C$, lifting to periodized curves inside $N_\RR$. Choose some point $p_0$ in the complement of the curves inside $N_\RR$. The degree $0$ line bundle $\O(\Gamma_1-\Gamma_0)$ in the dual torus $\Lambda^*_\RR/M$ is represented by the element $K$ of $\Lambda^*_\RR$ that maps a class $\lambda$ to the following: choose a path $\gamma$ between $p_0$ and $p_0+\lambda$, and add the moments of the oriented edges of $\Gamma_0$ and $\Gamma_1$ met by the path, oriented so that the intersection sign between $\gamma$ and the oriented edge is $-$ for $\Gamma_1$ and $+$ for $\Gamma_0$.
	\end{theo}
	
	\begin{rem}
	It is important to take the same base-point for lifting all the loops: during the computation, we assume that the slope of $f$ at $p_0$ is $0$, and this determines uniquely $f$, not up to addition of a monomial anymore.
	\end{rem}

	\begin{proof}
Up to addition of a monomial, we can assume that $f$ has slope $0$ at $p_0$. Let $\lambda$ be an element of $\Lambda$, lifted to a path between $p_0$ and $p_0+\lambda$ that intersects $\Gamma_0$ and $\Gamma_1$ transversally. Let $x_1,\dots,x_{n-1}$ be the intersection points between the path and the lifts of $\Gamma_0$ and $\Gamma_1$. Let also $x_0=p_0$ and $x_n=p_0+\lambda$. Let $m_i$ is the slope of $f$ in the region where the path between $x_i$ and $x_{i+1}$ lies. By assumption, $m_0=m_n=0$. We have that
\begin{align*}
K_\lambda = f(x_0+\lambda)-f(x_0) & = \sum_{i=0}^{n-1} f(x_{i+1})-f(x_i) \\
& = \sum_{i=0}^{n-1}\langle m_i,x_{i+1}-x_i\rangle, \\
& = \sum_{i=1}^{n} \langle m_{i-1},x_{i}\rangle - \sum_{i=0}^{n-1} \langle m_i,x_{i}\rangle\\
& = \sum_{i=1}^{n-1} \langle m_{i-1}-m_i,x_i\rangle. \\
\end{align*}
Furthermore, $x_i$ belongs to an edge $e$ of either $\Gamma_0$ or $\Gamma_1$ directed by $w_eu_e$. The monomial $m_i-m_{i-1}$ is equal to $\det(w_eu_e,-)$ when $e$ is oriented so that the intersection index $e\cdot\gamma$ at $x_i$ is positive. The result follows since $\Gamma_1$ appears positively in the divisor and $\Gamma_0$ negatively.
	\end{proof}
	
	Going around the same steps, one easily obtains a similar statement for phased-tropical varieties. In other words, one does not only care about the modulus of the numbers, but also their argument. This relates the line bundles $\O(\CCC_1-\CCC_0)$ to the moment of the edges of the curves. It suggests that this can be done in the complex setting as well: these can be expressed as the integrals of some well-chosen meromorphic forms generalizing the notion of \textit{moment}. 

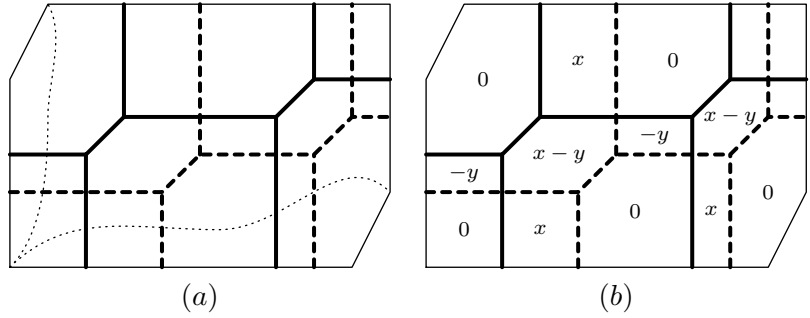
\begin{figure}
\begin{center}
\begin{tabular}{cc}
\begin{tikzpicture}[line cap=round,line join=round,>=triangle 45,x=0.5cm,y=0.5cm]
\draw [line width=0.5pt] (0,0)--++ (9,0)--++ (1,2)--++ (0,5)--++ (-9,0)--++ (-1,-2)--++ (0,-5);

\draw [line width=1.5pt,dashed] (0,2)--++ (4,0)--++ (1,1)--++ (3,0)--++ (1,1)--++ (1,0);
\draw [line width=1.5pt,dashed] (4,0)--++ (0,2);
\draw [line width=1.5pt,dashed] (8,0)--++ (0,3);
\draw [line width=1.5pt,dashed] (5,3)--++ (0,4);
\draw [line width=1.5pt,dashed] (9,4)--++ (0,3);

\draw [line width=1.5pt] (0,3)--++ (2,0)--++ (1,1)--++ (4,0)--++ (1,1)--++ (2,0);
\draw [line width=1.5pt] (2,0)--++ (0,3);
\draw [line width=1.5pt] (7,0)--++ (0,4);
\draw [line width=1.5pt] (3,4)--++ (0,3);
\draw [line width=1.5pt] (8,5)--++ (0,2);

\draw [line width=0.5pt,dotted,looseness=1] (0,0) to [in=180] (5.5,1) to [out=0] (10,2);
\draw [line width=0.5pt,dotted,looseness=0.8] (0,0) to [in=-90] (1,4) to [out=90, in=-60] (1,7);

\begin{scriptsize}

\end{scriptsize}
\end{tikzpicture} &
\begin{tikzpicture}[line cap=round,line join=round,>=triangle 45,x=0.5cm,y=0.5cm]
\draw [line width=0.5pt] (0,0)--++ (9,0)--++ (1,2)--++ (0,5)--++ (-9,0)--++ (-1,-2)--++ (0,-5);

\draw [line width=1.5pt,dashed] (0,2)--++ (4,0)--++ (1,1)--++ (3,0)--++ (1,1)--++ (1,0);
\draw [line width=1.5pt,dashed] (4,0)--++ (0,2);
\draw [line width=1.5pt,dashed] (8,0)--++ (0,3);
\draw [line width=1.5pt,dashed] (5,3)--++ (0,4);
\draw [line width=1.5pt,dashed] (9,4)--++ (0,3);

\draw [line width=1.5pt] (0,3)--++ (2,0)--++ (1,1)--++ (4,0)--++ (1,1)--++ (2,0);
\draw [line width=1.5pt] (2,0)--++ (0,3);
\draw [line width=1.5pt] (7,0)--++ (0,4);
\draw [line width=1.5pt] (3,4)--++ (0,3);
\draw [line width=1.5pt] (8,5)--++ (0,2);

\begin{scriptsize}
\draw (1,1) node {$0$};
\draw (3,1) node {$x$};
\draw (5.5,1.5) node {$0$};
\draw (7.5,1.5) node {$x$};
\draw (9,2) node {$0$};
\draw (1,2.5) node {$-y$};
\draw (3.5,3) node {$x-y$};
\draw (1.5,5) node {$0$};
\draw (4,5.5) node {$x$};
\draw (6.5,5.5) node {$0$};
\draw (6,3.5) node {$-y$};
\draw (8,4) node {$x-y$};
\end{scriptsize}
\end{tikzpicture} \\
$(a)$ & $(b)$ \\
\end{tabular}
\caption{\label{figure tropical curves linear system}Two tropical curves $\Gamma_0$ and $\Gamma_1$ of the same degree in a tropical abelian surface, with two paths generating $\Lambda$ on $(a)$, and with the slope of a function with corner locus $\Gamma_0\cup\Gamma_1$ on $(b)$.}
\end{center}
\end{figure}

	\begin{expl}
	In Figure \ref{figure tropical curves linear system} we can see two tropical curves of the same degree $\begin{pmatrix}
	1 & 0 \\
	0 & 2 \\
	\end{pmatrix}$, $\Gamma_0$ is dashed and $\Gamma_1$ is full. We also have depicted two paths that generate the homology of $\TT A$. We have a piecewise affine function $f$ whose corner locus is $\Gamma_0\cup\Gamma_1$: the function changes its slope convexly or concavely if it meets $\Gamma_1$ or $\Gamma_0$. Assuming the slope is $0$ at the bottom left corner of the hexagon, we get the slopes depicted on Figure \ref{figure tropical curves linear system} $(b)$. Theorem \ref{theorem relation moment linear system tropical} asserts that to recover the parameter from the quasi-periodicity of $f$, we only need to make the sum of the moment of the edges intersected by the dotted path on Figure \ref{figure tropical curves linear system} $(a)$.
	\end{expl}
	
	\begin{rem}
In the case of tropical curves  inside $N_\RR$, the tropical Menelaus theorem \cite{mikhalkin2017quantum} asserts that the sum of the moments of the unbounded ends is $0$. More generally, the moments of the edges intersecting a contractible loop is $0$ by balancing condition. In an abelian surface, there are non-contractible loops, and the balancing condition ensures that the sum of the moments, if defined, does not depend on the chosen representative for the loop. The linear system condition can be seen as imposing the sum of the moments for the loops inside $\TT A$.
\end{rem}

If one is given a tropical curve inside a torus, it is possible to deform cycles that it contains to get a different tropical curve of the same degree. Theorem \ref{theorem relation moment linear system tropical} asserts that to keep the deformation in the same linear system, it is only possible to deform cycles that are homologically trivial inside $H_1(\TT A,\ZZ)$, so that the sum of the moments remains constant.

	\subsubsection{General statement}
	
In the complex setting, the problem is that we cannot transpose verbatim the tropical proof since a complex curve does not have any edges, and a complex curve does not intersect a generic path for obvious dimensional reasons. The statement we are looking for is an analog of the following $1$-dimensional statement that relates a divisor on an elliptic curve to the element that it defines in its Picard group. We recall it with a proof, hoping the reader will then find the proof of the $2$-dimensional case more natural.

\begin{prop}
Let $\EEE=\CC^*/\langle \lambda\rangle$ be an elliptic curve and let $D=\sum_i n_i z_i$ be a degree $0$ divisor on $\EEE$, where $z_i$ are points in $\EEE$. The Picard group of $\EEE$ is isomorphic to $\EEE$, and we have $\O(D)=\prod z_i^{n_i}\in\EEE$.
\end{prop}

\begin{proof}
We consider an annulus $X=\{a\leqslant |z| <|\lambda|a\}$ that is a fundamental domain in $\CC^*$ for the quotient $\CC^*/\langle\lambda\rangle$. By abuse of notation, we also denote by $z_i$ the preimages in $\CC^*$ of the points $z_i\in\EEE$ that belong to the fundamental domain. Let $\sigma$ be a meromorphic function obtained as a quotient of the $\theta$-functions associated with the positive and negative part of $D$, as explained above in the $2$-dimensional setting. The function $\sigma$ satisfies the relation $\sigma(\lambda z)=K\sigma(z)$ for some $K\in\CC^*$, and we want to recover the constant $K\in\CC^*$. To do this, we consider the holomorphic form $\varphi=\frac{1}{2i\pi}\log\sigma(z)\frac{\dd z}{z}$ on $\CC^*\backslash \{\lambda^n z_i\}_{n,i}$. This $1$-form is not correctly defined since we lack a logarithm for $\sigma$. However, we can define it on any simply connected domain by taking such a logarithm.

Let $X_\varepsilon$ be the fundamental domain where we have removed small disks around the points $z_i$: $X_\varepsilon=X\backslash\bigcup_i D(z_i,\varepsilon)$. It has two boundary components coming from $X$, and one component $\partial D(z_i,\varepsilon)$ per point in the support of $D$. We now choose a cellular decomposition of $X_\varepsilon$ whose vertices on the boundary components are precisely $a$, $\lambda a$ and $z_i+\varepsilon$. For each face $F$, we can define a holomorphic $1$-form $\varphi_F=\frac{1}{2i\pi}\log\sigma(z)\frac{\dd z}{z}$ by choosing a logarithm for $\sigma$. This form is well-defined up to addition of $\frac{\dd z}{z}$. Using Stoke's formula, we have for each face
$$\int_{\partial F}\varphi_F=\int_F \dd\varphi_F=0,$$
since $\varphi_F$ is closed. Furthermore, each edge $E$ (with a chosen orientation) not in $\partial X_\varepsilon$ is on the boundary of two faces $F_+$ and $F_-$. Thus, there exists an integer $k_E\in\ZZ$ such that on $E$, $\varphi_{F_+}-\varphi_{F_-}=k_E\frac{\dd z}{z}$ since the two logarithms for $\sigma$ differ by an element of $2i\pi\ZZ$. We add the contributions for all the faces and regroup the integrals over each edge according to whether they lie on the boundary of $X_\varepsilon$ or not:
\begin{align*}
0=\sum_{F}\int_{\partial F}\varphi_F & =\sum_{E\subset\partial X_\varepsilon}\int_E \varphi_F + \sum_{E\nsubseteq \partial X_\varepsilon} k_E\int_E\frac{\dd z}{z}. \\
\end{align*}
Now, we observe the following facts:
\begin{itemize}[label=$\circ$]
\item First, for the two boundary components of $X_\varepsilon$ coming from $X$, one is the image of the other by multiplication by $\lambda$, but with opposite orientations as boundary of $X$. Moreover, as $\sigma(\lambda z)=K\sigma(z)$, the two integrals cancel each other up to the term $\frac{1}{2i\pi}\int_E \log K\frac{\dd z}{z}=\log K$, modulo $2i\pi\ZZ$ if the logarithms for $\sigma$ do not coincide.
\item The integrals $\int_{\partial D(z_i,\varepsilon)}\varphi_F$ over the other boundary components go to $0$ when $\varepsilon$ goes to $0$: as $\sigma$ has a zero or pole at each $z_i$, $\log\sigma(z)$ has a singularity whose modulus is in $O(\log\varepsilon)$, while the circle $\partial D(z_i,\varepsilon)$ has size $2\pi\varepsilon$. Hence, the integral goes to $0$.
\item It remains the contribution of the integrals $\int_E \frac{\dd z}{z}$. If we work modulo $2i\pi$, each oriented edge $E$ relates some points $z_E^-$ and $z_E^+$ of the decomposition. Thus, we have
$$\int_E \frac{\dd z}{z}\equiv \log z_E^+ -\log z_E^- \ \mathrm{mod}\ 2i\pi,$$
and
$$\sum_{E\nsubseteq \partial X_\varepsilon} k_E\int_E\frac{\dd z}{z}=\sum_{z}\left(\sum_{E\ni z}k_E\right)\log z \ \mathrm{mod}\ 2i\pi,$$
where the sum is over the vertices of the triangulation.
\begin{itemize}[label=-]
\item If such a point $z$ is not on the boundary, then we can find locally a logarithm for $\sigma$, and it implies that $\sum_{E\ni z}k_E=0$.
\item If $z=z_i+\varepsilon$ is the unique point lying on the boundary component $\partial D(z_i,\varepsilon)$, as $\sigma(z)$ has a pole of order $n_i$, the monodromy is $2i\pi n_i$, and we get that $\sum_{E\ni z}k_E=n_i$.
\item Finally, if $z=a$ is the point lying on the inner boundary component of $\partial X$, the other point being $\lambda a$. Let $k=\sum_{E\ni z}k_E$ and $k'=\sum_{E\ni \lambda z}k_E$. We have that $k+k'=0$ since we can find a logarithm for $\sigma$ locally at $a$ and transport it by $\lambda$ to get one at $\lambda a$. Thus, $k\log a+k'\log(\lambda a)=k'\log\lambda$.
\end{itemize}
\end{itemize}
As we have computed the integrals mod $2i\pi$, we get a relation mod $2i\pi$. The last point also tells to view it mod $\log\lambda$. Taking into account the above remarks, and the exponential of the relation modulo $2i\pi$, we get that
$$K=\prod_i z_i^{n_i} \ \mathrm{mod}\ \lambda,$$
which is the desired relation.
\end{proof}

The statement in the $2$-dimensional case is presented below. When necessary, let $z$ and $w$ be two coordinates functions on the complex torus $N_{\CC^*}$. The theorem relates the element defined in the Picard group by the divisor $\CCC_1-\CCC_0$ to some integrals over both curves. These integrals are an analog to the tropical notion of \textit{moment of an edge}. Let us consider a circle $\gamma$ on a curve $\CCC\subset N_{\CC^*}$. It realizes some class $n_\gamma\in N$. Thus, the monomial $\chi^{m_\gamma}$ corresponding to $m_\gamma=\det(n_\gamma,-)\in M$ admits a logarithm in a neighborhood of $\gamma$. There is a unique $m'_\gamma$ completing $m_\gamma$ into an oriented basis of $M$.

\begin{defi}
With the above notations, we define \textit{moment} of the circle is the scalar
$$\frac{1}{2i\pi}\int_{\gamma}\log\chi^{m_\gamma}\frac{\dd\chi^{m'_\gamma}}{\chi^{m'_\gamma}}.$$
\end{defi}

The $1$-forms that one integrates are not closed but they are holomorphic. The holomorphic part implies that two circles realizing the same homotopy class inside $\CCC$ have the same moment, since they bound a holomorphic curve on which the derivative of the $1$-form restricts to $0$. This derivative is $\frac{\dd\chi^m}{\chi^m}\wedge\frac{\dd\chi^{m'}}{\chi^{m'}}=\dzdw$, which is a holomorphic $2$-form on $N_{\CC^*}$. Furthermore, as it restricts to $0$ on any $\CCC$, the sum of the moments over cycles whose sum is homologically trivial inside $\CCC$ is also $0$.

For what follows, let $T_{x,y}=\mathrm{Log}^{-1}(x,y)$ be the tori of elements in $N_{\CC^*}$ having a fixed projection under $\mathrm{Log}:N_{\CC^*}\to N_\RR$. Let $(x_0,y_0)\in N_\RR$ be a fixed point such that $T_{x_0,y_0}$ does not intersect $\CCC_0$ or $\CCC_1$.

\begin{theo}\label{theorem relation moment linear system complex}
Using the above notations, for $\lambda\in\Lambda$, let $X$ be a cobordism between $T_{x_0,y_0}$ and $\lambda T_{x_0,y_0}$, it intersects $\CCC_0$ and $\CCC_1$ along circles $\gamma_i$, let $\mu_i$ be the moment of the circle $\gamma_i$ oriented such that $i$ times the tangent direction completes the tangent space to $X$ into an oriented basis of $N_{\CC^*}$. Then $K$ maps $\lambda$ to $e^{\sum\varepsilon_i\mu_i}$, where $\varepsilon_i=\pm 1$ according to whether $\gamma_i$ is on $\CCC_0$ or $\CCC_1$.
\end{theo}

\begin{rem}
The abelian surface $\CC A=N_{\CC^*}/\Lambda$ can be seen as a torus fibration over the tropical torus $N_\RR/\log|\Lambda|$ with fiber $N\otimes\RR/\ZZ$ taking the modulus coordinate by coordinate. The tori $T_{x,y}$ are fibers of this fibration, and one obtains a cobordism between two fibers by taking the preimage of a path in the base.
\end{rem}

\begin{proof}
We adapt the proof to the $2$-dimensional setting. Let $\sigma$ be the meromorphic function on $N_{\CC^*}$ that has poles and zeros along $\CCC_0$ and $\CCC_1$, obtained by quotienting suitable $\theta$-functions $\theta_0$ and $\theta_1$. For the reader's convenience, here is a correspondence between the objects used in the proof and the $1$-dimensional case:
\begin{itemize}[label=-]
\item The fundamental domain $X$ becomes now a cobordism between the two real tori $T_{x_0,y_0}=\mathrm{Log}^{-1}(x_0,y_0)$ and $\lambda T_{x_0,y_0}$, where $\mathrm{Log}:N_{\CC^*}\to N_\RR$ is the logarithm coordinates by coordinates. We can assume that neither of the two tori intersects $\CCC_0$ or $\CCC_1$. The cobordism $X$ is a $3$-dimensional manifold with boundary and not a fundamental domain anymore.
\item We use the not really defined $2$-form $\varphi=\frac{1}{(2i\pi)^2}\log\sigma(z,w)\frac{\dd z}{z}\wedge\frac{\dd w}{w}$. As before, this form is not defined because of the lack of global logarithm for $\sigma$, but it is correctly defined up to $2i\pi\ZZ$ on simply connected domains. Notice that if $\chi^m=z^aw^b$ and $\chi^{m'}=z^cw^d$ are two monomials such that $(m,m')$ is an oriented basis of $M$, then $\frac{\dd\chi^m}{\chi^m}\wedge\frac{\dd\chi^{m'}}{\chi^{m'}}=\dzdw$, and we can thus choose any basis of monomials. Moreover, $\log\chi^m\frac{\dd \chi^{m'}}{\chi^{m'}}$ is a primitive of the $2$-form on any domain where we can define a logarithm for $\chi^m$, and we have
$$\log\chi^m\frac{\dd \chi^{m'}}{\chi^{m'}}=\log z\frac{\dd w}{w}.$$
\item To get a simply connected domain, we first get rid of the points where it is not defined: $X\cap\CCC_0$ and $X\cap\CCC_1$. If $X$ is chosen generically, it is transversal to $\CCC_0$ and $\CCC_1$, and thus their intersection are $1$-dimensional varieties without boundary: a disjoint union of circles $\gamma_i$. Each of them realizes a class $n_i\in N$. We then consider $X_\varepsilon$ which is $X$ where we have removed a small $\varepsilon$-tubular neighborhood $U(\gamma_i,\varepsilon)$ of each $\gamma_i$. Finally, we take a cellular subdivision of $X_\varepsilon$ such that for each torus on the boundary of $X_\varepsilon$, the decomposition has $1$ vertex, $2$ edges and $1$ face. Furthermore, we assume that for the boundary components $\partial U(\gamma_i,\varepsilon)$, one of the edges realizes a loop contractible inside $N_{\CC^*}$.  This means one of the loops is in the ``small" direction of the real torus $\partial U(\gamma_i,\varepsilon)$. The picture to have in mind is $(S^1)^2\times [0;1]$, with some circles not meeting the boundary, and we remove small tubular neighborhoods of the circle, and choose a cellular decomposition.
\end{itemize}

For each maximal cell $V$ in the decomposition, we can define a holomorphic  $2$-form $\varphi_V=\frac{1}{(2i\pi)^2}\log\sigma\frac{\dd z}{z}\wedge\frac{\dd w}{w}$ by choosing a logarithm for $\sigma$. By Stoke's formula, we have that
$$\int_{\partial V}\varphi_V=\int_V \dd \varphi_V=0.$$
If $F$ is an oriented face that belongs to two differents maximal cells $V_+$ and $V_-$, we have that
$$\varphi_{V_+}-\varphi_{V_-}=\frac{1}{2i\pi} k_F\frac{\dd z}{z}\wedge\frac{\dd w}{w},$$
for some integer $k_F$. Adding all the contributions, we get that
$$0=\sum_{V}\int_{\partial V}\varphi_V=\sum_{F\subset\partial X_\varepsilon}\int_F \varphi_V + \sum_{F\nsubseteq\partial X_\varepsilon}\frac{1}{2i\pi} k_F\int_F \dzdw.$$
Now, we simplify the relation:
\begin{itemize}[label=$\circ$]
\item First, for the boundary components $T_{x_0,y_0}$ and $\lambda T_{x_0,y_0}$, the two logarithms for $\sigma$ differ by $\frac{1}{(2i\pi)^2}\int_{T_{x_0,y_0}}\log K_\lambda\dzdw=\log K_\lambda$.
\item Then, for the other boundary components, $\log\sigma$ has still a singularity whose modulus is in $O(\log\varepsilon)$, and thus the integral vanishes when $\varepsilon$ goes to $0$ since the torus has an area in $O(\varepsilon)$.
\item For the $\int_F \dzdw$, we relate them to integrals over the boundaries of the faces. It is indeed possible to find primitives for the $2$-form $\dzdw$: each $\log\chi^m\frac{\dd \chi^{m'}}{\chi^{m'}}$ is a solution whenever defined. We have
$$\int_F \dzdw=\int_{\partial F}\log\chi^m\frac{\dd \chi^{m'}}{\chi^{m'}}.$$
The logarithm is correctly defined up to an addition of $2i\pi$. Then,
\begin{align*}
\sum_{F\nsubseteq\partial X_\varepsilon} \frac{1}{2i\pi} k_F \int_F\dzdw= & \sum_{E\subset\partial X_\varepsilon} \frac{1}{2i\pi}\left(\sum_{F\supset E}k_F \right)\int_E \log\chi^m\frac{\dd \chi^{m'}}{\chi^{m'}} \\
& + \sum_{E\nsubseteq\partial X_\varepsilon} \frac{1}{2i\pi}\left(\sum_{F\supset E}k_F \right)\int_E \log\chi^m\frac{\dd \chi^{m'}}{\chi^{m'}}.\\
\end{align*}
For edges not contained in the boundary, as we can find a local logarithm for $\sigma$, we get that $\sum_{F\supset E}k_F=0$ as in the planar case. Concerning the edges on the boundary: the edges not going around the circles have a well-defined $\log z$, and $\log w$ as well, so the integral is $0$. Meanwhile, the integrals going along the circles contribute $\int_{X\cap\CCC_i}\log\chi^m\frac{\dd \chi^{m'}}{\chi^{m'}}$, which is precisely their moment. Last, the edges on $T_{x_0,y_0}$ and $\lambda T_{x_0,y_0}$ cancel their contributions as for edges not on the boundary.
\end{itemize}
\end{proof}

\begin{rem}
To some extent, we consider the integral
$$\frac{1}{(2i\pi)^2}\int_{T_{x,y}}\log\sigma(z,w)\dzdw.$$
It is a version with argument of the Ronkin function, where $\sigma$ would be replaced by its modulus $|\sigma|$, wo that there are no problem with the logarithm anymore. The cobordism $X_\varepsilon$ between $T_{x_0,y_0}$, $\lambda T_{x_0,y_0}$ and the $\partial U(\gamma_i,\varepsilon)$ gives a relation etween the integrals over these tori. Moreover, as $\sigma(\lambda\cdot(z,w))=K_\lambda\sigma(z,w)$, the integrals over the first two boundary components contribute for $\log K_\lambda$. However, as $\varphi$ is not globally defined, we have to restrict to simply connected domains and add up the Stoke's relations on each of them.
\end{rem}

The upshot of the preceding theorem is that it allows to express the fact that a curve belongs to a fixed linear system purely in terms of the curve. This is done by considering certain integrals on the curves, and it has a nice expression close to the tropical limit by imposing a condition on the phases of the edges of the curves.

\section{Enumerative problem and counts}

\subsection{Enumerative problem}

\subsubsection{Problem and setting.} We have seen that the moduli space of genus $g$ curves in a class $C\in\Lambda\otimes N$ is of dimension $g$. In the first paper, we studied the enumerative problems of counting curves passing through $g$ points. We now focus on the enumerative problem where we impose the line bundle as well.

\begin{rem}
In \cite{bryan1999generating}, J. Bryan and N. Leung prove a result that transform conditions imposed on the line bundle defined by the curve into conditions on the curve itself. More precisely, they consider the additional conditions imposed by $1$-cycles inside $H_1(\CC A,\ZZ)$. The fixed linear system condition is equivalent to meeting four $1$-cycles whose classes form a basis of $H_1(\CC A,\ZZ)$. Here, we prefer to adopt a different point of view by recovering the line bundle from the curve using Theorem \ref{theorem relation moment linear system complex}.
\end{rem}

Similarly to the case of genus $g$ curves passing through $g$ points, there might be reducible solutions if the class $C$ decomposes as a sum $C=C_1+C_2$, but it is possible to deform the irreducible components separately. Thus, we only care about irreducible solutions.

\medskip

Let $h:\Gamma_0\to\TT A$ be a fixed genus $2$ curve realizing the class $C$ and $\P\subset\TT A$ be a configuration of $g-2$ points chosen generically. We look for parametrized tropical curves $h:\Gamma\to\TT A$ of genus $g$ in the class $C$ passing through $\P$ and such that the divisor $\Gamma-\Gamma_0$ realizes a fixed element inside $\mathrm{Pic} A=\Lambda^*_\RR/M$.

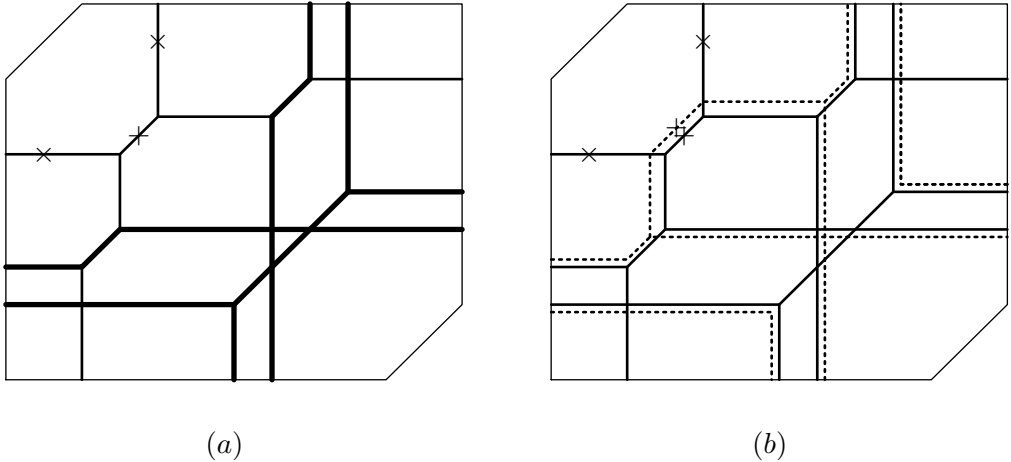
\begin{figure}
\begin{center}
\begin{tabular}{cc}
\begin{tikzpicture}[line cap=round,line join=round,>=triangle 45,x=0.5cm,y=0.5cm]
\clip(-1,-1) rectangle (12.5,10.5);
\draw [line width=0.5pt] (0,0)--++ (10,0)--++ (2,2)--++ (0,8)--++ (-10,0)--++ (-2,-2)--++ (0,-8);

\draw [line width=2pt] (0,3)--++ (2,0)--++ (1,1)--++ (2,0)--++ (7,0);
\draw [line width=2pt] (0,2)--++ (6,0)--++ (3,3)--++ (3,0);
\draw [line width=1pt] (0,6)--++ (3,0)--++ (1,1)--++ (3,0);
\draw [line width=1pt] (8,8)--++ (4,0);
\draw [line width=2pt] (6,0)--++ (0,2);
\draw [line width=2pt] (7,0)--++ (0,7);
\draw [line width=2pt] (8,8)--++ (0,2);
\draw [line width=1pt] (4,7)--++ (0,3);
\draw [line width=2pt] (7,7)--++ (1,1);
\draw [line width=2pt] (9,5)--++ (0,5);
\draw [line width=1pt] (2,0)--++ (0,3);
\draw [line width=1pt] (3,4)--++ (0,2);

\draw (1,6) node {$\times$};
\draw (4,9) node {$\times$};
\draw (3.5,6.5) node {$+$};

\end{tikzpicture} &
\begin{tikzpicture}[line cap=round,line join=round,>=triangle 45,x=0.5cm,y=0.5cm]
\clip(-1,-1) rectangle (12.5,10.5);
\draw [line width=0.5pt] (0,0)--++ (10,0)--++ (2,2)--++ (0,8)--++ (-10,0)--++ (-2,-2)--++ (0,-8);

\draw [line width=1pt] (0,3)--++ (2,0)--++ (1,1)--++ (2,0)--++ (7,0);
\draw [line width=1pt] (0,2)--++ (6,0)--++ (3,3)--++ (3,0);
\draw [line width=1pt] (0,6)--++ (3,0)--++ (1,1)--++ (3,0);
\draw [line width=1pt] (8,8)--++ (4,0);
\draw [line width=1pt] (6,0)--++ (0,2);
\draw [line width=1pt] (7,0)--++ (0,7);
\draw [line width=1pt] (8,8)--++ (0,2);
\draw [line width=1pt] (4,7)--++ (0,3);
\draw [line width=1pt] (7,7)--++ (1,1);
\draw [line width=1pt] (9,5)--++ (0,5);
\draw [line width=1pt] (2,0)--++ (0,3);
\draw [line width=1pt] (3,4)--++ (0,2);

\draw (1,6) node {$\times$};
\draw (4,9) node {$\times$};
\draw (3.5,6.5) node {$+$};

\draw (3.3,6.7) node {$+$};

\draw [line width=1pt,dotted] (0,1.8)--++ (5.8,0)--++ (0,-1.8);
\draw [line width=1pt,dotted] (0,3.2)--++ (2,0) --++ (0.6,0.6)-- (12,3.8);
\draw [line width=1pt,dotted] (2.6,3.8)-- (2.6,6)--++ (1.4,1.4)--++ (3.2,0)--++(0.6,0.6)--++(0,2);
\draw [line width=1pt,dotted] (12,5.2)--++ (-2.8,0)--++ (0,4.8);
\draw [line width=1pt,dotted] (12,5.2)--++ (-2.8,0)--++ (0,4.8);
\draw [line width=1pt,dotted] (7.2,0)--++ (0,7.4);

\end{tikzpicture} \\
$(a)$ & $(b)$ \\
\end{tabular}
\caption{\label{figure curve point and linear system} A genus $5$ tropical curve passing through $3$ points and belonging to a fixed linear system. On $(a)$, the complement of the marked points retracts on the bold part, and on $(b)$, a small deformation in the same linear system obtained by slightly moving one of the marked points.}
\end{center}
\end{figure}

\begin{expl}
On Figure \ref{figure curve point and linear system} we have drawn a genus $5$ curve of degree $\begin{pmatrix}
3 & 0 \\
0 & 3 \\
\end{pmatrix}$ passing through $3$ points. The complement of the marked points retracts on a subgraph of genus $2$ which is in bold on $(a)$. In the case of curves passing through $g$ points, the complement of marked points is without cycle, but not here.

\smallskip

On $(b)$ we draw a small deformation of the curve in the same linear system obtained by slightly moving one of the marked point. In the case of curves passing through $g$ points, the deformation is recovered by deforming the unique loop that contains the moving marked point and not the other. In this case, the linear system conditions expressed by Theorem \ref{theorem relation moment linear system tropical} only allows the deformation of cycles which are trivial inside $H_1(\TT A,\ZZ)$. As there is a three dimensional space of cycles containing the moving marked point, we get a unique cycle trivial in $H_1(\TT A,\ZZ)$, which might not be a circle, and we can deform it, as depicted on $(b)$.
\end{expl}

In the case of curves passing through $g$ points, we know that the complement of the marked points is a tree. Here, we have the following proposition.

\begin{prop}\label{proposition form of solutions}
Let $h:\Gamma\to\TT A$ be a curve of genus $g$ in the class $C$ passing through $\P$ and such that $\Gamma$ is linearly equivalent to $\Gamma_0$. Then:
\begin{itemize}[label=-]
\item the curve $\Gamma$ is simple and rigid,
\item the complement $\Gamma\backslash\P$ is connected, of genus $2$, and the map $H_1(\Gamma\backslash\P,\ZZ)\to H_1(\TT A,\ZZ)$ is injective.
\end{itemize}
\end{prop}

\begin{proof}
The first statement comes from dimensional considerations proven in the first paper. If a tropical curve is not simple, either $h$ is an immersion and $\Gamma$ is not trivalent, or $h$ is not an immersion, and then it is possible to reparametrize $h(\Gamma)$ by a new map $h':\Gamma'\to\TT A$ such that $h'$ is an immersion, but $\Gamma'$ is of smaller genus or has vertices that are not trivalent. In either case, the dimension of the deformation space of $\Gamma'$ is strictly smaller than $g$, and its combinatorial type does not provide any solution if the constraints are chosen generically.

The rigidity comes from the fact that for generic constraints, there is a finite number of solutions: if there was a $1$-parameter family of solutions, we would get at least one non simple tropical curve solution to the problem, contradicting the genericity of the constraints.

The map $H_1(\Gamma\backslash\P,\ZZ)\to H_1(\TT A,\ZZ)$ is injective: if it contained a cycle homologous to $0$, it would be possible to deform it and get a $1$-parameter family of solutions, contradicting the genericity of the constraints. Thus, $b_1(\Gamma\backslash\P)\leqslant 2$. As $\chi(\Gamma\backslash\P)=1-g+g-2=-1$, we get that $b_0(\Gamma\backslash\P)=b_1(\Gamma\backslash\P)-1\leqslant 1$. Hence, the complement of the marked points is connected, and of genus $2$ with the map to $H_1(\TT A,\ZZ)$ injective.
\end{proof}

Proposition \ref{proposition form of solutions} allows one to describe more precisely the form of the solutions. It says that for each solution $h:\Gamma\to\TT A$, there is a distinguished subgraph $\Sigma\subset\Gamma$ on which the complement of the marked points retracts. This is a graph of genus $2$. Should the curve only be subject to the point constraints, it would be possible to deform this subgraph. It is fixed using the linear system constraints and Theorem \ref{theorem relation moment linear system tropical}.

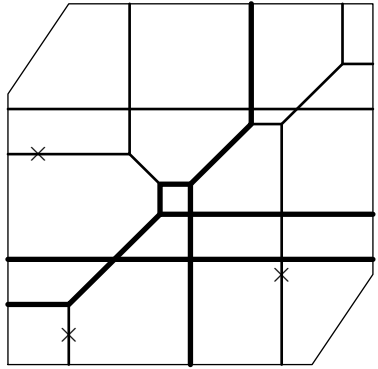
\begin{figure}
\begin{center}
\begin{tikzpicture}[line cap=round,line join=round,>=triangle 45,x=0.4cm,y=0.4cm]
\draw [line width=0.5pt] (0,0)--++ (10,0)--++ (2,3)--++ (0,9)--++ (-10,0)--++ (-2,-3)--++ (0,-9);

\draw [line width=2pt] (0,2)--++ (2,0)--++ (3,3)--++ (7,0);
\draw [line width=2pt] (0,3.5)--++ (12,0);

\draw [line width=2pt] (6,0)--++ (0,6)--++ (2,2)--++ (0,4);

\draw [line width=2pt] (5,5)--++ (0,1)--++ (1,0);

\draw [line width=1pt] (0,7)--++ (4,0)--++(0,5);
\draw [line width=1pt] (5,6)--++ (-1,1);
\draw [line width=1pt] (2,0)--++ (0,2);
\draw [line width=1pt] (9,0)--++ (0,8);
\draw [line width=1pt] (8,8)--++ (1,0)--++(2,2)--++(1,0);
\draw [line width=1pt] (11,10)--++ (0,2);
\draw [line width=1pt] (0,8.5)--++ (12,0);

\draw (2,1) node {$\times$};
\draw (1,7) node {$\times$};
\draw (9,3) node {$\times$};

\end{tikzpicture}
\caption{\label{figure tropical curve point linear system dumbbell} A genus $5$ tropical curve in a fixed linear system passing through $3$ points.}
\end{center}
\end{figure}

\begin{expl}
On Figure \ref{figure curve point and linear system} $(a)$ we have a genus $5$ tropical curve passing through $3$ points in a fixed linear system, and the graph $\Sigma$ drawn in bold. It is a Theta graph, but it can also be a dumbbell graph, as depicted on Figure \ref{figure tropical curve point linear system dumbbell}.
\end{expl}

\subsubsection{Multiplicities.} Recall from the first paper the two following multiplicities for tropical curves.

\begin{defi}
Let $h:\Gamma\to\TT A$ be a simple tropical curve. We define
\begin{itemize}[label=$\circ$]
\item the usual multiplicity $m_\Gamma=\prod_V m_V$, where $m_V=|\det(a_V,b_V)|$ is the absolute value of the determinant between two out of the three outgoing slopes for each trivalent vertex,
\item the refined multiplicity $m^q_\Gamma=\prod_V [m_V]_q=\prod_V \frac{q^{m_V/2}-q^{-m_V/2}}{q^{1/2}-q^{-1/2}}\in\ZZ[q^{\pm 1/2}]$, where the vertex multiplicity $m_V$ is replaced by its quantum analog.
\end{itemize}
\end{defi}

In the case of curves passing through $g$ points, although the correspondence theorem gives a multiplicity equal to $\delta_\Gamma m_\Gamma$, we can prove that the multiplicity $m_\Gamma$ as well as the refined multiplicity $m_\Gamma^q$ also provide invariants. This is proved by studying deformation of the tropical curves and walls of the enumerative problem. Unfortunately, these multiplicities do not provide an invariant in general for the enumeration of curves in a fixed linear system due to the appearance of walls of a new kind. As the correspondence theorem from the next section shows, there is an extra term appearing in the multiplicity. This extra-term is necessary for the invariance.

\begin{defi}
Let $h:\Gamma\to\TT A$ be a simple tropical curve and $\Sigma\subset\Gamma$ a genus $2$ subgraph. We define $\Lambda_\Gamma^\Sigma$ to be the index of $H_1(\Sigma,\ZZ)$ inside $H_1(\TT A,\ZZ)\simeq\Lambda$.
\end{defi}

\begin{expl}
Let $(\lambda_1,\lambda_2)$ be the basis of $H_1(\TT A,\ZZ)\simeq\Lambda$ obtained by taking a horizontal loop and a vertical loop. All the vertices of the curve on Figure \ref{figure curve point and linear system} have multiplicity $1$, and the two cycles generating the homology of the bold part $\Sigma$ realize the classes $2\lambda_1$ and $2\lambda_2$, so that $\Lambda_\Gamma^Sigma=4$.  Meanwhile, for the curve on Figure \ref{figure tropical curve point linear system dumbbell}, the cycles realize the homology classes $2\lambda_1$ and $\lambda_2$, so that one has $\Lambda_\Gamma^\Sigma=2$.
\end{expl}

The multiplicity provided by the correspondence theorem uses $\Lambda_\Gamma^\Sigma$: the multiplicity of a curve passing through $\P$ depends on the subgraph on which the complement of the marked points retracts. We now introduce the following enumerative counts:
\begin{align*}
N_{g,C,k}^{FLS}(\TT A,\P) = \sum_{\substack{h(\Gamma)\supset\P \\ \delta(\Gamma)=k}} \Lambda_\Gamma^\Sigma m_\Gamma \in\NN , \\
BG_{g,C,k}^{FLS}(\TT A,\P) = \sum_{\substack{h(\Gamma)\supset\P \\ \delta(\Gamma)=k}} \Lambda_\Gamma^\Sigma  m^q_\Gamma \in\ZZ[q^{\pm 1/2}] .\\
\end{align*}
counting curves with a fixed gcd, and then the following counts for curves without a gcd condition:
\begin{align*}
M_{g,C}^{FLS}(\TT A,\P) &  = \sum_{h(\Gamma)\supset\P} \Lambda_\Gamma^\Sigma  m_\Gamma = \sum_{k|\delta(C)} N_{g,C,k}^{FLS}(\TT A,\P) \in\NN , \\
N_{g,C}^{FLS}(\TT A,\P) &  = \sum_{h(\Gamma)\supset\P} \delta_\Gamma \Lambda_\Gamma^\Sigma  m_\Gamma = \sum_{k|\delta(C)}k N_{g,C,k}^{FLS}(\TT A,\P) \in\NN , \\
R_{g,C}^{FLS}(\TT A,\P) & = \sum_{h(\Gamma)\supset\P} \delta_\Gamma \Lambda_\Gamma^\Sigma  m^q_\Gamma \in \ZZ[q^{\pm 1/2}] , \\
BG_{g,C}^{FLS}(\TT A,\P) & = \sum_{h(\Gamma)\supset\P} \Lambda_\Gamma^\Sigma  m^q_\Gamma = \sum_{k|\delta(C)} BG_{g,C,k}^{FLS}(\TT A,\P) \in \ZZ[q^{\pm 1/2}] . \\
\end{align*}

As in the case of curves passing through $g$ points, we prove the invariance for the counts of curves with fixed gcd, and the invariance for all the other counts follows. Only the count $N_{g,C}^{FLS}(\TT A,\P)$ has a meaning in teh complex setting as the count $\N_{g,C}^{FLS}$ of genus $g$ complex curves passing through $g-2$ points in a fixed linear system. The meaning of the refined counts remains open although already known interpretation of these refined invariants in the toric setting can probably be adapted.

\subsection{Correspondence theorem}
\label{section correspondence}

In \cite{nishinou2020realization}, T. Nishinou gives a realization theorem that describes the possible deformation of a tropical curve close to the tropical limit, and then uses this realization theorem to prove a correspondence theorem and compute the number of complex passing through $g$ points with a given tropicalization. In this section, we adapt the proof of the correspondence theorem to compute the number of tropical curves passing through $g-2$ points and belonging to a fixed linear system. This uses Theorem \ref{theorem relation moment linear system complex}.

\subsubsection{Tropical evaluation map.} Let $h:\Gamma\to\TT A$ be a simple tropical curve passing through $\P$ and belonging to the preferred linear system. We denote by $\Gamma'$ the graph whose vertices are the vertices of $\Gamma$ and the marked points. We consider the following map:
$$\Psi: \bigoplus_{V\in V(\Gamma')} N \to \bigoplus_{e\in E(\Gamma')}N/N_e\oplus\bigoplus_1^{g-2} N\oplus \Lambda^*.$$
Let $(\phi_V)\in\bigoplus_{V\in V(\Gamma')}N$.
\begin{itemize}[label=$\circ$]
\item The coordinates in $N/N_e$ are obtained as the difference of the coordinates for the two adjacent vertices: $\phi_{\partial^+ e}-\phi_{\partial^- e}$.
\item The coordinates in $\bigoplus_1^{g-2} N$ are the coordinates of the marked points: $\phi_{V_i}$. These are the same as for the map $\Psi$ considered in the first paper and in \cite{nishinou2020realization}.
\item The coordinates on $\Lambda^*$ are related to Theorem \ref{theorem relation moment linear system tropical} and obtained as follows: for a basis $\lambda_1,\lambda_2$ of $\Lambda$, choose loops in $\TT A$ that represent the classes and that are transverse to $\Gamma$. For $\lambda_1$ (resp. $\lambda_2$), let $E_i$ be the oriented edges intersected by $\lambda_1$ (resp. $\lambda_2$), $V_i$ their source, and $u_i$ their slope. Then, do the sum $\sum_i \det(u_i,\phi_{V_i})$. In other words, it is the sum of the moments of the edges met by a loop representing the homology class. A different choice of path lifting the loop leads to a different map $\Psi$, obtained by a change of basis on the codomain of $\Psi$ using balancing condition.
\end{itemize}

\begin{rem}
Technically, the moment depends on the edge, and not on the vertices adjacent to it. However, as the domain of the map corresponds to the position of the vertices, we take one of the vertices adjacent to the edge to compute its moment. The map corresponding to another choice of vertex, or another choice of loop is obtained by adding or substracting components from $N/N_e$.
\end{rem}

The domain of $\Psi$ has rank
$$2(2g-2+g-2)=6g-8,$$
and its codomain has rank 
$$3g-3+g-2+2(g-2)+2=6g-7.$$
The map cannot be surjective. In fact, as in the first paper, we still have the relation among the coordinates  of $\bigoplus_e N/N_e$, corresponding to the Menelaus relation on the curve obtained by the lifting procedure.

\subsubsection{Complex setting and correspondence.} Let $\CC A_t$ be a family of complex abelian surfaces as considered in the second section, tropicalizing to $\TT A$. We assume that there is a fixed linear system with curves in it, and we consider a configuration of $g-2$ points $\P_t\subset\CC A_t$ that tropicalizes to a generic point configuration $\P\subset\TT A$. We know that the number of genus $g$ complex curves in the linear system passing through $\P_t$ is equal to $\N_{g,C}^{FLS}$. We now state the correspondence theorem.

\begin{theo}\label{theorem correspondence}
Let $h:\Gamma\to \TT A$ be a parametrized tropical curve passing through $\P$ and in the fixed linear system. The number of genus $g$ complex curves in the fixed linear system passing through $\P_t$ and that tropicalize to $h:\Gamma\to\TT A$ is
$$m_\Gamma^\CC=|\ker\Psi_{\CC^*}|\prod_{e\in E(\Gamma')}w_e.$$
In particular, we have that $N_{g,C}^{FLS}(\TT A,\P)$ does not depend on $\P$ and $\TT A$ as long as these choices are generic, and $N_{g,C}^{FLS}=\N_{g,C}^{FLS}$.
\end{theo}

\begin{rem}
It should also be possible to adapt the proof of the correspondence theorem from \cite{nishinou2020realization} to include the constraints coming from $1$-cycles and compute the invariants introduced in \cite{bryan1999generating}.
\end{rem}

Theorem \ref{theorem correspondence} is proven in section \ref{section proof correspondence}.

\subsection{Statement of the results}

\subsubsection{Multiplicity formula.} The correspondence theorem gives a formula for the multiplicity with which to count the tropical curves solution to the tropical enumerative problem so that their count yields the complex invariant $\N_{g,C}^{FLS}$. The following theorem gives an more concrete expression of this multiplicity.

\begin{theo}\label{theorem multiplicity formula}
Let $h:\Gamma\to\TT A$ be a parametrized tropical curve in the fixed linear system that passes through the generic configuration $\P$. Let $\Sigma\subset\Gamma$ be the genus $2$ subgraph on which $\Gamma\backslash\P$ retracts. The multiplicity $m^\CC_\Gamma$ splits as the following product:
$$m^\CC_{\Gamma}=\delta_\Gamma \Lambda_\Gamma^\Sigma m_\Gamma.$$
\end{theo}

The multiplicity is comprised of three terms: the gcd of the weights of the edges, the product of vertex multiplicities, and a new additional term that is equal to the index of the $H_1(\Gamma\backslash\P,\ZZ)$ inside $\Lambda$. This term appears because unlike the case of curves passing through $g$ points, the point conditions are not enough to determine the phases of the edges of the curves.

\begin{expl}
As all the vertices are of multiplicity $1$, the curves on Figure \ref{figure curve point and linear system} $(a)$ and Figure \ref{figure tropical curve point linear system dumbbell} have respective multiplicities $4$ and $2$.
\end{expl}

\subsubsection{Invariance statements.} The correspondence theorem relates the count of tropical curves using multiplicity $\delta_\Gamma\Lambda_\Gamma^\Sigma m_\Gamma$ to the complex invariants from \cite{bryan1999generating} and \cite{bryan2018curve}. Thus, the invariance of the count $N_{g,C}^{FLS}(\TT A,\P)$ follows from the invariance on the complex side. We now study the invariance using the refined multiplicity from a tropical point of view. First, we have an invariance statement regarding the choice of the point configuration $\P$.

\begin{theo}\label{theorem point invariance}
The refined count $BG_{g,C,k}^{FLS}(\TT A,\P)$ of genus $g$ curves in the class $C$ with fixed gcd passing through $\P$ in the fixed linear system does not depend on the choice of $\P$ and the line bundle as long as it is generic.
\end{theo}

We then have the following corollary.

\begin{coro}
The counts $N_{g,C,k}^{FLS}(\TT A,\P)$, $M_{g,C}^{FLS}(\TT A,\P)$, $N_{g,C}^{FLS}(\TT A,\P)$, $R_{g,C}^{FLS}(\TT A,\P)$ and $BG_{g,C}^{FLS}(\TT A,\P)$ do not depend on the choice of $\P$ and the line bundle as long as it is generic.
\end{coro}

We remove $\P$ from the notation to denote the associated invariant. Then, we have an invariance statement regarding the choice of the abelian surface $\TT A$.

\begin{theo}\label{theorem surface invariance}
The refined invariant $BG_{g,C,k}^{FLS}(\TT A)$ does not depend on the choice of $\TT A$ as long as it is chosen generically among the surfaces that contain curves in the class $C$.
\end{theo}

\begin{coro}
The counts $N_{g,C,k}^{FLS}(\TT A)$, $M_{g,C}^{FLS}(\TT A)$, $N_{g,C}^{FLS}(\TT A)$, $R_{g,C}^{FLS}(\TT A)$ and $BG_{g,C}^{FLS}(\TT A)$ do not depend on the choice of $\TT A$ as long as it is generic among the abelian surfaces with curves in the class $C$.
\end{coro}

\begin{rem}
The refinement of the vertex multiplicities is a standard trick for tropical curves. The term $\delta_\Gamma$ could be replaced with anything and we would still have invariance since we have invariance for the counts of curves with fixed gcd. It would be interesting to find a way to also refine the term $\Lambda_\Gamma^\Sigma $.
\end{rem}

\section{Proof of the results}

\subsection{Proof of the correspondence theorem}
\label{section proof correspondence}

The proof follows the steps of the proof of the correspondence theorem from \cite{nishinou2020realization}. We refer to it for more details, as we prefer to focus on the changes that we need to make.

\medskip

We consider the family of abelian surfaces $\CC A_t$, completed with the central fiber as $t$ goes to $\infty$. The central fiber $\CC A_\infty$ is a union of toric surfaces that meet along their toric divisors. This family of surfaces is obtained as follows. Let $\Xi$ be the subdivision of $N_\RR$ induced by the periodic lifts of all the tropical curves solution to the enumerative problem. Then, consider the fan in $N_\RR\times\RR_+$ over the subdivision of $\Xi\times\{1\}$. This fan possesses an action of $\Lambda$ by translation, which induces an action on the almost toric variety associated to the fan. The projection onto the last coordinate makes it into a family $\CC A_t$ with central fiber $\CC A_\infty$ as described. See the first paper or \cite{nishinou2020realization} for more details. The proof of the correspondence theorem uses log-geometry, and is as follows. First, we find the prelog curves solution to the enumerative problem. These are certain nodal curves in $\CC A_\infty$ that satisfy specific conditions. After transforming them into log-curves, we see that there is a unique way to deform them into solutions inside the abelian surfaces $\CC A_t$.

\begin{proof}[Proof of Theorem \ref{theorem correspondence}]
\begin{itemize}[label=$\circ$]
\item Let $h:\Gamma\to\TT A$ be a tropical curve solution to the tropical enumerative problem. Let $\Gamma'$ be $\Gamma$ with additional bivalent vertices at the place of marked points, and $\widehat{\Gamma}$ the curve with the subdivision induced by the subdivision $\Xi$ of $N_\RR$. We have the following maps:
{ \everymath={\displaystyle}
$$\begin{array}{rccl}
F: & \bigoplus_{V\in V(\Gamma)}N  & \longrightarrow & \bigoplus_{E\in E(\Gamma)}N/N_e, \\
\Psi: & \bigoplus_{V\in V(\Gamma')}N & \longrightarrow & \bigoplus_{E\in E(\Gamma')}N/N_e\oplus\bigoplus_1^{g-2}N\oplus\Lambda^*, \\
\widehat{F}: & \bigoplus_{V\in V(\widehat{\Gamma})}N  & \longrightarrow & \bigoplus_{E\in E(\widehat{\Gamma})}N/N_e, \\
\widehat{\Psi}: & \bigoplus_{V\in V(\widehat{\Gamma})}N & \longrightarrow & \bigoplus_{E\in E(\widehat{\Gamma})}N/N_e\oplus\bigoplus_1^{g-2}N\oplus\Lambda^*. \\
\end{array}$$
}

The maps $F$ and $\widehat{F}$ are already introduced in \cite{nishinou2020realization}, while the maps $\Psi$ and $\widehat{\Psi}$ play the role analog to $G$ and $\widehat{G}$ from \cite{nishinou2020realization}. The definition of these map is as the definition of $\Psi$ in Section \ref{section correspondence}. The domain of $F$ is of rank $2(2g-2)=4g-4$, while its codomain is of rank $3g-3$.Thus, we have that
$$\dim\ker F_\RR=g-1+\dim\mathrm{Coker}F_\RR.$$
Furthermore, we have $\dim\ker F_\RR=g$ and $\dim\mathrm{Coker}F_\RR=1$. In fact, the vector space $\ker F_\RR$ describes the infinitesimal deformations of the curve $h:\Gamma\to\TT A$, and its dimension was already computed to be $g$.

\item Recall that a prelog curve is a map $\varphi_\infty:\CCC_\infty\to\CC A_\infty$ from a nodal curve $\CCC_\infty$ to the central fiber $\CC A_\infty$ that satisfies some conditions. Every component $\CCC_V$ of $\CCC_\infty$ is a mapped to some irreducible component $\CC A_V$ of $\CC A_\infty$. As a curve inside the toric variety $\CC A_V$, $\CCC_V$ needs to be torically transverse. Moreover, for each intersection point with the toric boundary of $\CC A_V$ with tangency order $k$, let $\CC A_W$ be the irreducible component of $\CC A_\infty$ that shares the same divisor, then there is an irreducible component $\CCC_W$ of $\CCC_\infty$ mapped to $\CC A_W$ having an intersection point of the same tangency order $k$ at the same point.

\medskip

According to Proposition 5.3 of \cite{nishinou2020realization}, if $\varphi_\infty:\CCC_\infty\to\CC A_\infty$ is a prelog curve associated with the tropical curve $h:\Gamma\to\TT A$ and is equipped with a log-structure, then the logarithmic normal sheaf $\N=\varphi_\infty^*\Theta_{\CC A_\infty}/\Theta_{\CCC_\infty}$ (where $\Theta_\bullet$ denotes the logarithmic tangent sheaf) has derived global section $R\Gamma(\N)$ isomorphic to
$$\bigoplus_{V\in V(\Gamma)}N_\CC\to\bigoplus_{e\in E(\Gamma)}(N/N_e)_\CC.$$
In particular, $H^0(\CCC_\infty,\N)\simeq\ker F_\CC$ and $H^1(\CCC_\infty,\N)\simeq\mathrm{Coker}F_\CC$. This sheaf describes the deformations of the curve.

\item The vector space $\ker F_\RR$ describes infinitesimal deformations of $h:\Gamma\to\TT A$. As the curves solution to the enumerative problem are rigid, the map
$$\ker F_\RR\to\bigoplus_1^{g-2} N\oplus\Lambda^*_\RR,$$
is injective. As the spaces are of the same dimension, it follows that it is an isomorphism. We then deduce that
$$\ker \widehat{F}_\RR\to\bigoplus_1^{g-2} N\oplus\Lambda^*_\RR,$$
is surjective.

\item We have the point constraints $\P_t\subset\CC A_t$. They have been chosen so that each point $p_t^i\in\CC A_t$ has a limit $p_\infty^i\in\CC A_\infty$ that belongs to the main strata of some irreducible component $\CC A_V$. We also have the fixed linear system. We now look at the set $\SSS_\mathrm{prelog}$ of isomorphism classes of prelog curves $\varphi_\infty:\CCC_\infty\to\CC A_\infty$ associated with the fixed tropical solution $h:\Gamma\to\TT A$ that are solution to the enumerative problem. We now show that there are $|\ker\Psi_{\CC^*}|$ such curves. In fact, this set is a $\ker\Psi_{\CC^*}$-torsor.

\medskip

To prove this, we can proceed as in the proof of Proposition 5.7 in \cite{nishinou2020realization}. Alternatively, the data of a genus $g$ prelog-curve in $\SSS_\mathrm{prelog}$ consists in the data of a rational curve $\CCC_V$ in the component $\CC A_V$ attached to every vertex $V$ of $\Gamma$. To be of the right genus, these curve have to satisfy the following conditions:
	\begin{itemize}[label=-]
	\item If $V$ is a bivalent vertex, the curve $\CCC_V$ is rational and has exactly two intersection points with the toric boundary.
	\item If $V$ is a trivalent vertex, the curve $\CCC_V$ is rational and has exactly three intersection points with the toric boundary.
	\end{itemize}
	
\begin{rem}
Notice that in a component $\CC A_V$ associated with a ``quadrivalent vertex" of $\Gamma$ resulting from the crossing of two branches of $\Gamma$, we have exactly two components $\CCC_V$ and $\CCC_{V'}$ inside $\CC A_V=\CC A_{V'}$. Each of them is a rational curve that has exactly two intersection points with the toric boundary.
\end{rem}	

Moreover, the curves $\CCC_V$ have to satisfy the following conditions:
	\begin{itemize}[label=-]
	\item Two curves $\CCC_V$ and $\CCC_W$ for adjacent vertices $V$ and $W$ have the same intersection point with the toric divisor shared by $\CC A_V$ and $\CC A_W$, so that we have indeed a prelog curve.
	\item The curve $\CCC_{V_i}$ in the component $\CC A_{V_i}$ associated to a marked point $p^i$ passes through the limit point $p_\infty^i$.
	\item Concerning the linear system condition, we use Theorem \ref{theorem relation moment linear system complex}, that imposes the following constraint. Recall that we have chosen a basis $(\lambda_1,\lambda_2)$ of $\Lambda$, that are represented by loops $(\gamma_1,\gamma_2)$ inside $\TT A$. Up to deformation, we can assume that both loops are transverse to $\Gamma$, and that the intersection points belong to the subdivision $\Xi$, up to a refinement. Then, each intersection point between $\gamma_1$ or $\gamma_2$ and $\Gamma$ occurs at a bivalent vertex $Q$ associated to a rational curve $\CCC_Q$. Theorem \ref{theorem relation moment linear system complex} expresses the linear system condition as a condition on the sum of certain integrals over specific cycles inside the curve. In our case, these cycles belong to the components $\CCC_Q$ for $Q\in\Gamma\cap\gamma_i$. Morover, we can assume that $\CC A_Q\simeq\CC P^1\times\CC P^1$, and $\CCC_Q$ is the image of $z\mapsto(\alpha,z^w)\in(\CC P^1)^2$ for some $\alpha$ and $w$. Then, the above integral involved in Theorem \ref{theorem relation moment linear system complex} is just $\alpha^w$.
	\end{itemize}

\begin{rem}
This last point is the main difference with the proof from \cite{nishinou2020realization}: it replaces two of the point conditions by some other conditions involving positions of the components of the curves. It here intervenes when looking for the prelog curves solution to the problem, and it also intervenes when we deform them into true solutions in the last step of the proof.
\end{rem}

In each $\CC A_V$ associated to a trivalent vertex, as the curve $\CCC_V$ is rational of fixed degree with three intersection points with the toric boundary, it is unique up to multiplication by an element of $N_{\CC^*}$. Similarly for bivalent vertices and curves $\CCC_{V_i}$ associated to marked points, while for the other bivalent vertices, curves are unique up to multiplication by an element of $(N/N_e)_{\CC^*}$. Thus, the space of all curves is parametrized by $\bigoplus_{V\in V(\Gamma')}N_{\CC^*}$. Using Theorem \ref{theorem relation moment linear system complex}, the compatibility, point and linear system conditions mean that the point corresponding to a choice of $(\CCC_V)_V$ is mapped to a chosen $(1,p_\infty^i,l)$ by $\Psi_{\CC^*}$, hence the result.

\item Keeping the same notations, let $\varphi_\infty:\CCC_\infty\to\CC A_\infty$ be a given prelog curve of $\SSS_\mathrm{prelog}$ endowed with a log-structure. The set of sections $H^0(\CCC_\infty,\N)$ of the normal sheaf is identified with $\ker F_\CC$, \textit{i.e.} a section is determined by its restriction to each irreducible component $\CCC_V$ of $\CCC_\infty$. Furthermore, each restriction consists in the data of a vector in $N_\CC$ for every vertex of $\Gamma'$, or in $(N/N_e)_\CC$ for other bivalent vertices, and these vectors have to satisfy some compatibility condition: they belong to $\ker F_\CC$. Thus, the map
$$\Phi:H^0(\CCC_\infty,\N)\to \bigoplus_1^{g-2} \Theta_{\CC A_\infty}(p_i)/\varphi_{\infty*}\Theta_{\CCC_\infty}(p_i) \oplus\Lambda_{\CC}^*,$$
that maps a section of the log-normal sheaf $\N$ to its evaluation at the marked points, and the variation of the integrals induced by Theorem \ref{theorem relation moment linear system complex} that it induces ($\Lambda_\CC^*$ is indeed the tangent space to the Picard group), is an isomorphism: it is identified with $\ker F_\CC\to\bigoplus_1^{g-2} \Theta_{\CC A_\infty}(p_i)/\varphi_{\infty*}\Theta_{\CCC_\infty}(p_i) \oplus\Lambda_{\CC}^*$, which has been proven to be bijective.

\item Using Proposition 7.1 from \cite{nishinou2006toric}, for each prelog curve in $\SSS_\mathrm{prelog}$, there are precisely $\prod_{e\in E(\Gamma')}w_e$ possible log-structures on it.

\item To conclude, we show that it is possible to deform any log-curve to a nearby solution. Let $\varphi_\infty:\CCC_\infty\to\CC A_\infty$ is a log-curve that matches the constraints. In particular, $\CCC_\infty$ is a curve over $\CC$. Let $\Bbbk_n=\CC\left[\frac{1}{t}\right]/\left(\frac{1}{t^{n+1}}\right)$. It is the ring of germs of functions up to order $n$ near $\infty$. Then for any $n$ there exists a unique lift $\varphi_n:\CCC^{(n)}\to\CC A_t$ of $\varphi_\infty:\CCC_\infty\to\CC A_\infty$ on $\Bbbk_n$. This is proved by induction.
	\begin{itemize}[label=-]
	\item The fact holds for $n=0$ because the curve $\varphi_\infty:\CCC_\infty\to\CC A_\infty$ is a solution over $\Bbbk_0$.
	\item If a lift $\varphi_n:\CCC^{(n)}\to\CC A_t$ over $\Bbbk_n$ has been constructed, \cite{nishinou2020realization} ensures that the set of lifts over $\Bbbk_{n+1}$ is non-empty and is a torsor over $H^0(\CCC_\infty,\N)$. Using the constraints and the fact that $\Phi$ is bijective shows that there exists a unique lift that matches the constraints, finishing the proof.
	\end{itemize}
\end{itemize}
\end{proof}

\begin{rem}
In a sense, the proof relies on Hensel's lemma, as we show that we can find and count order $0$ solutions to the enumerative problem, and that there is a unique way to deform each of them. The non-degeneracy condition is ensured by the bijectivity of the evaluation map.
\end{rem}

\begin{rem}
One other way to proceed would be to adapt the proof from \cite{shustin2002patchworking}. As the linear system is fixed, any curve is the zero-locus of a section of the associated line bundle that is of the form $\theta^t(z)=\sum_{m_0}a_{m_0}(t)\theta^t_{m_0}(z)$. As we have seen that the tropical limit of $\theta_{m_0}^t$ is the tropical $\theta$-function $\Theta_{m_0}$, we can do as in \cite{shustin2002patchworking}, considering the $\theta$-function as monomials. The tropical geometry part recovers the order $0$ solutions, and a use of the implicit function theorem (or equivalently Hensel's lemma) yields solutions for $t\neq 0$.
\end{rem}

\subsection{Multiplicity formula}

\begin{proof}[Proof of Theorem \ref{theorem multiplicity formula}]
The idea is to compute $|\ker\Psi_{\CC^*}|$ in terms of the matrix of $\Psi$. As in the case of the first paper, taking bases of the domain and codomain, we have a short exact sequence
$$0 \to \ZZ^{6g-8}\xrightarrow{\Psi}\ZZ^{6g-7}\to \ZZ\oplus G\to 0,$$
where $G$ is a finite abelian group. Tensoring with $\CC^*$, we have that $\ker\Psi_{\CC^*}\simeq \mathrm{Tor}(G,\CC^*)$, whose cardinal is computed as the gcd of the maximal minors of the matrix of $\Psi$. Thus, we compute the gcd of the maximal minors of $\Psi$. As there is a relation among the coordinates of $\bigoplus_e N/N_e$, the only nonzero maximal minors are obtained by forgetting one of these coordinates.

Let $e_0$ be an edge of $\Gamma'$. As usual, we compute the determinant by making successive developments with respect to rows and columns of the determinant.
\begin{itemize}[label=$\circ$]
\item First, we develop with respect to the rows $\bigoplus_1^{g-2}N$ of the evaluation at the marked points. Each marked point brings a block
$$\begin{array}{|c|}
\hline
\det(u_e,-) \\
\hline
I_2 \\
\hline
\det(u_e,-) \\
\hline
\end{array},$$
where columns are the coordinates of the marked point, middle rows coordinates of its evaluation, and the other rows evaluation for the adjacent edges. The middle rows are the only non-zero elements in these rows. Thus, we can develop with respect to them.
\item The complement of the marked point is a genus $2$ graph with branches attached to it. The leafs of these branches correspond to the marked points. The first step was to get rid of the marked points. The next step is to get rid of the branches. Let $V$ be a vertex adjacent to two leafs. Let $u_1,u_2$ and $u_e$ be the primitive slopes of the adjacent edges. The columns of $V$ bring a block
$$\begin{array}{|c|}
\hline
\det(u_e,-) \\
\hline
\det(u_1,-) \\
\hline
\det(u_2,-) \\
\hline
\end{array}.$$
The last two rows are the only non-zero elements in the rows. Thus, we can develop with respect to these rows, thus pruning the branch. We get the determinant for the graph where the vertex $V$ has been deleted replaced by an open edge, multiplied by $\det(u_1,u_2)=\frac{m_V}{w_1w_2}$.
\item In the first paper, for curves passing through $g$ points, the previous step was enough to completely compute the determinant: as the complement of the marked point is a tree, it is completely pruned. Here, this step is not enough because the complement of the marked points is not a tree. This is emphasized by the fact that at some point, we have to develop with respect to the rows corresponding to $\Lambda^*$. These rows are obtained by looking at the intersection between $\Gamma$ and two loops realizing a basis of $H_1(\TT A,\ZZ)\simeq\Lambda$. Each of these loops intersect various edges of $\Gamma$. For each intersected edge, we have two possibilities:
	\begin{itemize}[label=-]
	\item If the edge belongs to a branch of $\Gamma\backslash\P$, it has been pruned. Thus, the columns have been deleted during the second step.
	\item If the edge belongs to the genus $2$ subgraph, let $V$ be one of its vertices, \textit{i.e.} one that is not adjacent to a branch. We deform the loops so that their intersection points with the subgraph all belong to the edges immediately adjacent to $V$. This is possible by shifting the intersection point along the edges of the subgraph, and deforming the loop, remaining close to $\Gamma$. This may add intersection points with the edges immediately adjacent to the genus $2$ subgraph, but as their columns have been deleted in the previous step, this does not add new coefficients.
	\end{itemize}
Then, using the particular form presented above, the only vertex that has non-zero coordinates in the $\Lambda^*$ component is $V$. Thus, we can develop with respect to the last two rows.

The block that appears is corresponds to a map $N\to\Lambda^*$. Assume that $\Sigma$ is a Theta graph (The dumbbell case is treated similarly) Let $E_1$ and $E_2$ be the two edges adjacent to $V$ that contain intersection points between $\Sigma$ and the paths $\lambda_1$, $\lambda_2$. They give two loops $\gamma_1$ and $\gamma_2$ that span $H_1(\Sigma,\ZZ)$: $\gamma_1$ is the unique loop that contain $E_1$ and not $E_2$. Let $u_i$ be the slope of $u_i$. Let $(e_1,e_2)$ be a basis of $N$. Moving $V$ in the direction $e_i$ changes the moment of the edge $E_j$ by $\det(u_j,e_i)$. Moreover, the path $\lambda_k$ has $\gamma_j\cdot\lambda_k$ intersection points with the edge $E_j$. Thus, the matrix is
\begin{align*}
 & \begin{pmatrix}
(\gamma_1\cdot\lambda_1)\det(u_1,e_1)+(\gamma_2\cdot\lambda_1)\det(u_2,e_1) & (\gamma_1\cdot\lambda_1)\det(u_1,e_2)+(\gamma_2\cdot\lambda_1)\det(u_2,e_2) \\
(\gamma_1\cdot\lambda_2)\det(u_1,e_1)+(\gamma_2\cdot\lambda_2)\det(u_2,e_1) & (\gamma_1\cdot\lambda_2)\det(u_1,e_2)+(\gamma_2\cdot\lambda_2)\det(u_2,e_2)
\end{pmatrix} \\
= & \begin{pmatrix}
\gamma_1\cdot\lambda_1 & \gamma_2\cdot\lambda_1 \\
\gamma_1\cdot\lambda_2 & \gamma_2\cdot\lambda_2 \\
\end{pmatrix}\begin{pmatrix}
\det(u_1,e_1) & \det(u_1,e_2) \\
\det(u_2,e_1) & \det(u_2,e_2) \\
\end{pmatrix}.
\end{align*}
The determinant of the first matrix is $[H_1(\Sigma,\ZZ):H_1(\TT A,\ZZ)]=\Lambda_\Gamma^\Sigma$, and the second is $m_V$. Thus, the determinant id multiplied by $\Lambda_\Gamma^\Sigma m_V $, and the vertex $V$ is deleted.

\item After the previous step, the graph has be opened at the vertex $V$, and we can now apply the second step to prune the remaining vertices. In the end, we get
$$\frac{w_{e_0}}{\prod_{e\in E(\Gamma')}w_e}\Lambda_\Gamma^\Sigma \prod_V m_V.$$
Taking the gcd and multiplying by $\prod_{e\in E(\Gamma')}w_e$ yields the result.
\end{itemize}

\end{proof}

\subsection{Independence of the choice of the points}

\begin{proof}[Proof of Theorem \ref{theorem point invariance}]
Up to translation, we can assume that the line bundle is fixed once and for all. Thus, we only have to move the point configuration $\P$. Moreover, we can move the points one at a time. If the point configuration is chosen generically, the curves are simple, and a small deformation of the point configuration translates to a small deformation of the curve. 

Assume a unique point $p$ moves. In the case of curves passing through $g$ points, the deformation of the curve is obtained by moving the unique cycle passing through the moving point and not the other marked points. Here, there are several cycles passing through the moving point and avoiding the other marked points since $\Gamma\backslash\P$ contains a subgraph $\Sigma$ of genus $2$. However, these deformations may change the linear system.

Let $\Sigma_p$ be the smallest connected subgraph of $\Gamma$ containing $\Sigma$ and $p$. It is of genus $3$, and the map $H_1(\Sigma_p,\ZZ)\to H_1(\TT A,\ZZ)$ has a rank $1$ kernel. The deformation is here obtained by deforming the unique cycle that realizes the zero class inside $H_1(\TT A,\ZZ)\simeq\Lambda$. We have the following possible walls that can occur:
\begin{itemize}[label=$\circ$]
\item A quadrivalent vertex appears. There are three adjacent combinatorial types, and we have local invariance as proven in \cite{itenberg2013block}.

\item A cycle is contracted to a segment: we have two quadrivalent vertices linked by two parallel edges. This is handled as in \cite{itenberg2013block}: the combinatorial type is the same on both sides of the wall, except the marked point has changed sides.

\item A marked point meets a trivalent vertex $V$. Usually, this kind of wall is handled by the fact that the complement of the marked points has no cycle, and thus the marked point can only belong to two out of the three edges adjacent to $V$ since one of them would yield a disconnected $\Gamma\backslash\P$: there are only two combinatorial types and they have the same multiplicity. This is where the new wall of our situation appears.

\begin{itemize}[label=-]
\item First, if the trivalent vertex $V$ belongs to a branch of $\Gamma\backslash\P$, \textit{i.e.} does not belong to $\Sigma$, then there are only two out of the three adjacent combinatorial types that lead to solutions since for the third the complement of the marked points is not connected, as explained above.

\item Now, assume a marked point meets a trivalent vertex $V$ belonging to the distinguished genus $2$ subgraph. There are three adjacent combinatorial types, and unlike the planar situation, or the case of curves passing through $g$ points, all three might provide solutions. In this case, as $\delta_\Gamma$ and $m_\Gamma^q$ are the same for all adjacent combinatorial types, the invariance comes from the factor $\Lambda_\Gamma^\Sigma $.

Let $\Sigma_0$ be the subgraph of $\Gamma$ where we have removed $\P$ except the marked points that merges with $V$, and $\Sigma_1$, $\Sigma_2$ and $\Sigma_3$ the three genus $2$ subgraphs where we have removed one of the edge adjacent to $V$. We have inclusions $\Sigma_i\subset\Sigma_0\subset\Gamma$, giving injections
$$H_1(\Sigma_i,\ZZ)\hookrightarrow H_1(\Sigma_0,\ZZ)\hookrightarrow H_1(\Gamma,\ZZ).$$
As these $\Sigma_i$ is a subgraph of $\Sigma_0$, the map $H_1(\Sigma_i,\ZZ)\hookrightarrow H_1(\Sigma_0,\ZZ)$ is an inclusion, but its image is in fact a saturated sublattice: there exists $\varphi_i\in H^1(\Sigma_0,\ZZ)$ such that $H_1(\Sigma_i,\ZZ)=\ker\varphi_i\subset H_1(\Sigma_0,\ZZ)$. The cocycle $\varphi_i$ is in fact the map that measures the flow passing through the $i$-th edge adjacent to $V$. In particular, with a suitable choice of sign, we have the relation $\varphi_1+\varphi_2+\varphi_3=0$. This can be interpreted as a relation between the Pl\"ucker vectors of the sublattices $H_1(\Sigma_i,\ZZ)$ inside $H_1(\Sigma_0,\ZZ)$. Taking the other Pl\"ucker vector, we obtain the following relation:
$$\alpha_1\wedge\beta_1+\alpha_2\wedge\beta_2+\alpha_3\wedge\beta_3=0,$$
where $\alpha_i$ and $\beta_i$ are a basis of $H_1(\Sigma_i,\ZZ)$. Then, the index $\Lambda_\Gamma^{\Sigma_i}$ is obtained by taking $|\det(\alpha_i,\beta_i)|$, where $\alpha_i$ and $\beta_i$ are by abuse of notation the classes they represent inside $H_1(\TT A,\ZZ)\simeq\Lambda$. Thus, we have a relation
$$\Lambda_\Gamma^{\Sigma_1}+\Lambda_\Gamma^{\Sigma_2}=\Lambda_\Gamma^{\Sigma_3},$$
up to a relabeling of the edges. This relation corresponds indeed to the invariance of the tropical count near the wall.
\end{itemize}
\end{itemize}
\end{proof}

\begin{rem}
For the new kind of wall, the relation between the lattice indices $\Lambda_\Gamma^{\Sigma_i}$ can be refined in the following way: the intersection $\ker\varphi_1\cap\ker\varphi_2\cap\ker\varphi_3\subset H_1(\Sigma_0,\ZZ)$ is of rank $1$. It means that we can choose a cycle $\alpha\in H_1(\Sigma_0,\ZZ)$ that does not contain the vertex $V$. Then, we complete $\alpha$ with $\beta_i$ into a basis of $H_1(\Sigma_i,\ZZ)$. The invariance then comes from some relation $\beta_1+\beta_2=\beta_3$ up to $\alpha$.
\end{rem}

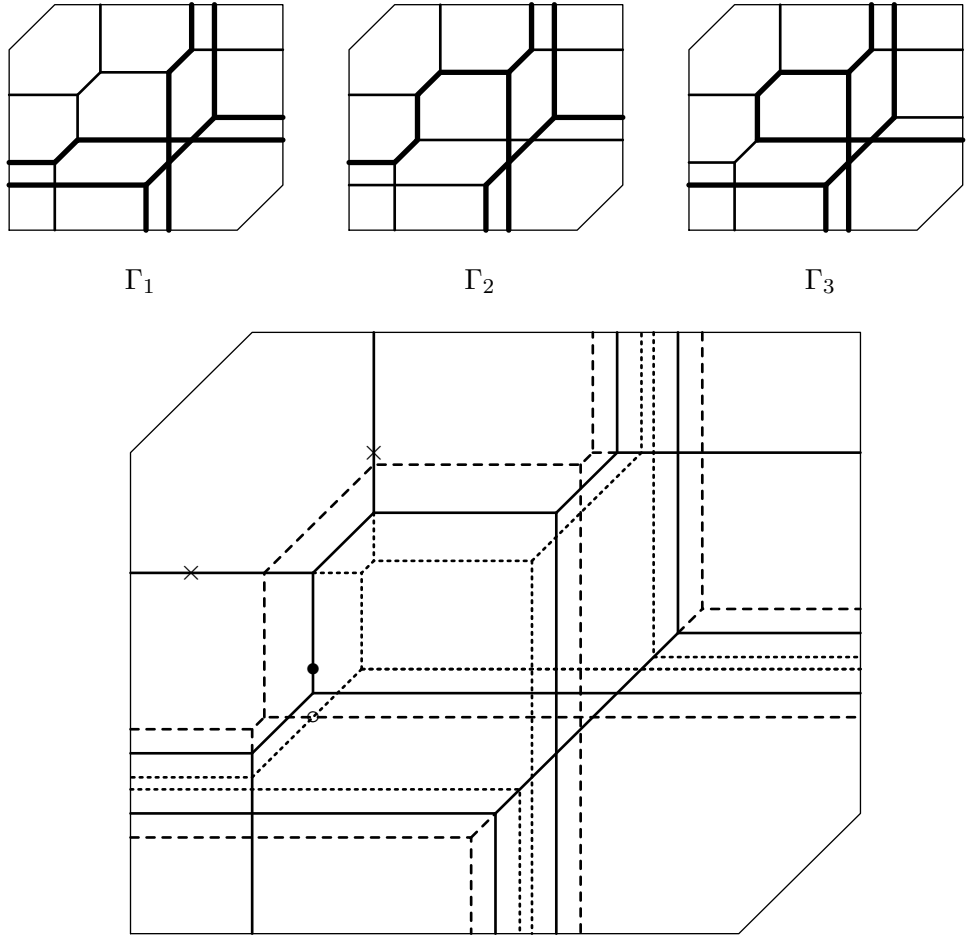
\begin{figure}[h]
\begin{center}
\begin{tabular}{ccc}
\begin{tikzpicture}[line cap=round,line join=round,>=triangle 45,x=0.3cm,y=0.3cm]
\clip(-1,-1) rectangle (12.5,10.5);
\draw [line width=0.5pt] (0,0)--++ (10,0)--++ (2,2)--++ (0,8)--++ (-10,0)--++ (-2,-2)--++ (0,-8);

\draw [line width=2pt] (0,3)--++ (2,0)--++ (1,1)--++ (2,0)--++ (7,0);
\draw [line width=2pt] (0,2)--++ (6,0)--++ (3,3)--++ (3,0);
\draw [line width=1pt] (0,6)--++ (3,0)--++ (1,1)--++ (3,0);
\draw [line width=1pt] (8,8)--++ (4,0);
\draw [line width=2pt] (6,0)--++ (0,2);
\draw [line width=2pt] (7,0)--++ (0,7);
\draw [line width=2pt] (8,8)--++ (0,2);
\draw [line width=1pt] (4,7)--++ (0,3);
\draw [line width=2pt] (7,7)--++ (1,1);
\draw [line width=2pt] (9,5)--++ (0,5);
\draw [line width=1pt] (2,0)--++ (0,3);
\draw [line width=1pt] (3,4)--++ (0,2);
\end{tikzpicture} &
\begin{tikzpicture}[line cap=round,line join=round,>=triangle 45,x=0.3cm,y=0.3cm]
\clip(-1,-1) rectangle (12.5,10.5);
\draw [line width=0.5pt] (0,0)--++ (10,0)--++ (2,2)--++ (0,8)--++ (-10,0)--++ (-2,-2)--++ (0,-8);

\draw [line width=2pt] (0,3) --++ (2,0) --++(1,1) --++ (0,2) --++(1,1)--++(3,0);
\draw [line width=1pt] (3,4)--++ (9,0);
\draw [line width=1pt] (0,2)--++ (6,0);
\draw [line width=2pt] (6,2)--++ (3,3)--++ (3,0);
\draw [line width=1pt] (0,6)--++ (3,0)--++ (1,1)--++ (3,0);
\draw [line width=1pt] (8,8)--++ (4,0);
\draw [line width=2pt] (6,0)--++ (0,2);
\draw [line width=2pt] (7,0)--++ (0,7);
\draw [line width=2pt] (8,8)--++ (0,2);
\draw [line width=1pt] (4,7)--++ (0,3);
\draw [line width=2pt] (7,7)--++ (1,1);
\draw [line width=2pt] (9,5)--++ (0,5);
\draw [line width=1pt] (2,0)--++ (0,3);
\draw [line width=1pt] (3,4)--++ (0,2);
\end{tikzpicture} &
\begin{tikzpicture}[line cap=round,line join=round,>=triangle 45,x=0.3cm,y=0.3cm]
\clip(-1,-1) rectangle (12.5,10.5);
\draw [line width=0.5pt] (0,0)--++ (10,0)--++ (2,2)--++ (0,8)--++ (-10,0)--++ (-2,-2)--++ (0,-8);

\draw [line width=1pt] (0,3) --++ (2,0) --++(1,1);
\draw [line width=2pt] (3,4)--++ (0,2) --++(1,1)--++(3,0);
\draw [line width=2pt] (3,4)--++ (9,0);
\draw [line width=2pt] (0,2)--++ (6,0);
\draw [line width=2pt] (6,2)--++ (3,3);
\draw [line width=1pt] (9,5)--++ (3,0);
\draw [line width=1pt] (0,6)--++ (3,0)--++ (1,1)--++ (3,0);
\draw [line width=1pt] (8,8)--++ (4,0);
\draw [line width=2pt] (6,0)--++ (0,2);
\draw [line width=2pt] (7,0)--++ (0,7);
\draw [line width=2pt] (8,8)--++ (0,2);
\draw [line width=1pt] (4,7)--++ (0,3);
\draw [line width=2pt] (7,7)--++ (1,1);
\draw [line width=2pt] (9,5)--++ (0,5);
\draw [line width=1pt] (2,0)--++ (0,3);
\draw [line width=1pt] (3,4)--++ (0,2);
\end{tikzpicture} \\
$\Gamma_1$ & $\Gamma_2$ & $\Gamma_3$ \\
\end{tabular}

\begin{tikzpicture}[line cap=round,line join=round,>=triangle 45,x=0.8cm,y=0.8cm]
\clip(-1,-1) rectangle (12.5,10.5);
\draw [line width=0.5pt] (0,0)--++ (10,0)--++ (2,2)--++ (0,8)--++ (-10,0)--++ (-2,-2)--++ (0,-8);

\draw [line width=1pt] (0,3)--++ (2,0)--++ (1,1)--++ (2,0)--++ (7,0);
\draw [line width=1pt] (0,2)--++ (6,0)--++ (3,3)--++ (3,0);
\draw [line width=1pt] (0,6)--++ (3,0)--++ (1,1)--++ (3,0);
\draw [line width=1pt] (8,8)--++ (4,0);
\draw [line width=1pt] (6,0)--++ (0,2);
\draw [line width=1pt] (7,0)--++ (0,7);
\draw [line width=1pt] (8,8)--++ (0,2);
\draw [line width=1pt] (4,7)--++ (0,3);
\draw [line width=1pt] (7,7)--++ (1,1);
\draw [line width=1pt] (9,5)--++ (0,5);
\draw [line width=1pt] (2,0)--++ (0,3);
\draw [line width=1pt] (3,4)--++ (0,2);

\draw (1,6) node {$\times$};
\draw (4,8) node {$\times$};
\draw (3,4.4) node {$\bullet$};
\draw (3,3.6) node {$\circ$};

\draw [line width=1pt,dotted] (3.8,4.4) --++ (0,1.6) --++ (0.2,0.2)-- (6.6,6.2)--(6.6,0);
\draw [line width=1pt,dotted] (3.8,4.4) -- (2,2.6) -- (0,2.6);
\draw [line width=1pt,dotted] (3.8,4.4) -- (12,4.4);
\draw [line width=1pt,dotted] (0,2.4) -- (6.4,2.4) -- (6.4,0);
\draw [line width=1pt,dotted] (6.6,6.2) -- (8.4,8) --(8.4,10);
\draw [line width=1pt,dotted] (3.8,4.4) -- (12,4.4);
\draw [line width=1pt,dotted] (8.6,10) -- (8.6,4.6) -- (12,4.6);
\draw [line width=1pt,dotted] (3,6) -- (3.8,6);
\draw [line width=1pt,dotted] (4,7) -- (4,6.2);

\draw [line width=1pt,dashed] (0,3.4) -- (2,3.4) --++ (0.2,0.2) -- (12,3.6);
\draw [line width=1pt,dashed] (2,3.4)--(2,3) ;
\draw [line width=1pt,dashed] (2.2,3.6) -- (2.2,6) -- (4,7.8) -- (7.4,7.8) -- (7.6,8) -- (7.6,10) ;
\draw [line width=1pt,dashed] (7.6,8) -- (8,8) ;
\draw [line width=1pt,dashed] (7.4,7.8) -- (7.4,0) ;  
\draw [line width=1pt,dashed] (0,1.6) -- (5.6,1.6) -- (5.6,0) ;
\draw [line width=1pt,dashed] (5.6,1.6) -- (6,2) ;
\draw [line width=1pt,dashed] (9.4,10) -- (9.4,5.4) -- (12,5.4) ;
\draw [line width=1pt,dashed] (9,5) -- (9.4,5.4) ; 

\end{tikzpicture}

\caption{\label{figure new kind of wall} Three tropical curves adjacent to the same wall. The three tropical curves with their graph $\Sigma$ are depicted in the first row. Then, we draw all solutions on the same picture to see which solution is on which side of the wall.}
\end{center}
\end{figure}

\begin{expl}
On Figure \ref{figure new kind of wall} we have three tropical curves that are adjacent to the same wall. When moving the marked point ``$\bullet$" to the marked point ``$\circ$", the solutions keep being in the same combinatorial type. The deformation is obtained by moving the vertex near the marked point in the direction $\RR(2,1)$. The line of slope $(2,1)$ splits the edges adjacent to the vertex in accordance to the sides of the wall. One checks that we have the relation $\Lambda_\Gamma^{\Sigma_1}=\Lambda_\Gamma^{\Sigma_2}+\Lambda_\Gamma^{\Sigma_3}$: $4=2+2$.
\end{expl}

\subsection{Independence of the choice of the surface}

\begin{proof}[Proof of Theorem \ref{theorem surface invariance}]
We proceed similarly as for the invariance for the curves passing through $g$ points in the first paper. Let $C\in \Lambda\otimes N$ be a class. Up to a change of basis of both $N$ and $\Lambda$, we can assume that $C$ is of the form $\begin{pmatrix}
d & 0 \\
0 & dn \\
\end{pmatrix}=d C_0$. The condition on $S=\begin{pmatrix}
s_{11} & s_{12} \\
s_{21} & s_{22} \\ 
\end{pmatrix}$ is that $ns_{12}=s_{21}$. We choose a generic path $S_t=(s_{ij}(t))_{ij}$ between two generic choices of matrices $S_0$ and $S_1$ such that this relation is satisfied along the path. A matrix $S$ is said to be generic if
$$\{ C'\in \Lambda\otimes N \text{ such that }C'S^T\in\S_2(\RR)\}=\ZZ C_0.$$
The path $S_t$ might have to cross the set of non-generic matrices to go from $S_0$ to $S_1$. Let $\Lambda_t=\mathrm{Im}S_t$ and $\TT A_t=N_\RR/\Lambda_t$.

\medskip

In addition to the path of matrices, for each $t$ we can choose a configuration of points $\P_t$ such that $\P_t$ varies continuously. Moreover, we can assume that for each $t$, the configuration $\P_t$ is generic as a point configuration inside $\TT A_t$ so that we know that the curves passing through $\P_t$ are simple. Notice that the choice $\P_t$ is generic inside $\TT A_t$ even when the matrix $S_t$ is not generic. Such a path of configurations exists because for every matrix $S$, the set of generic configurations is a dense open subset.

\medskip

Using the cutting process from the first paper, curves inside $\TT A_t$ are obtained from curves inside $N_\RR$ satisfying some gluing condition: a condition on the sum of moments between pairs of unbounded ends. If we impose furthermore to pass through $\P_t$ and the belonging to a fixed linear system, the curves solutions to the enumerative problem are in bijection with the curves inside $N_\RR$ passing through some fixed points (the lift of $\P_t$) and whose ends satisfy some moment conditions (both for the gluing and the linear system using Theorem \ref{theorem relation moment linear system tropical}). When moving $t$, the solutions move. To prove the invariance, we only need to show that for each value $t_*$, on a neighborhood of $t_*$, the solutions to the enumerative problem deform and give refined counts that are equal.

\medskip

If a curve solution to the problem in $\TT A_{t_*}$ is irreducible, it is possible to deform it and get nearby solutions for nearby $t$. Otherwise, the curve is reducible and $\TT A_{t_*}$ is not generic. We get solutions by deforming the reducible solutions for $\TT A_{t_*}$, smoothing some of the intersection points between the irreducible components. Given a set of intersection points to smooth, we see that there are two ways to deform the curve, one giving a solution for $\TT A_{t_*-\varepsilon}$, and one for $\TT A_{t_*+\varepsilon}$. The multiplicity is the same for both deformations of the curve, yielding the invariance.
\end{proof}

\section{Example}

The computation of explicit values of the invariants requires to have a way of solving the enumerative. Such an algorithm is provided in the third installment of the series. Nevertheless, we give two examples.

\begin{figure}[h]
\begin{center}
\begin{tikzpicture}[line cap=round,line join=round,>=triangle 45,x=0.5cm,y=0.5cm]
\clip(-0.5,-0.5) rectangle (21,5);

\draw [line width=1pt] (0,0)-- ++(19.8,0) --++ (0.2,1.2)-- ++(0,2.8)-- ++(-19.8,0) --++ (-0.2,-1.2)-- ++(0,-2.8);

\draw [line width=1.5pt] (0,1.5)-- (4,1.5)-- ++(0.2,0.2)-- (9,1.7)-- ++(0.6,0.6)-- (16,2.3)-- ++(0.2,0.4)-- (20,2.7);

\draw [line width=1.5pt] (4,0)-- (4,1.5) ++(0.2,0.2)-- (4.2,4);
\draw [line width=1.5pt] (9,0)-- (9,1.7) ++(0.6,0.6)-- (9.6,4);
\draw [line width=1.5pt] (9.2,0)-- ++(0,4);
\draw [line width=1.5pt] (9.4,0)-- ++(0,4);
\draw [line width=2pt] (16,0)-- (16,2.3) ++(0.2,0.4)-- (16.2,4);

\draw (16.2,3.5) node[right] {$2$};

\end{tikzpicture}

\caption{\label{figure stretched curve} A curve of genus $4$ in a long hexagon. It has one horizontal loop, and three vertical loops. Two of them make only one vertical round while the middle one makes three.}
\end{center}
\end{figure}
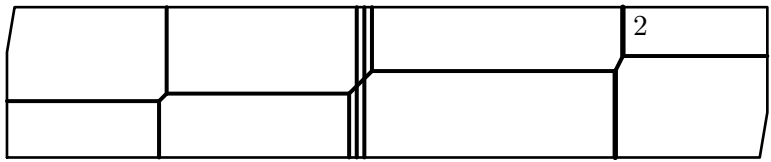

\begin{expl}
Assume that the class is $\begin{pmatrix}
1 & 0 \\
0 & n \\
\end{pmatrix}$ and that we are looking for curves of genus $g$. Assuming that the abelian surface is represented by a fundamental domain that is a long hexagon and that the points constraints are horizontally spread, all the curves are of the following form: a loop that goes around the horizontal direction, and $g-1$ loops that go around the vertical direction. Among the vertical loops, $g-2$ of them contain a marked point, and the position of the last one (denoted by $\infty$) is determined by the linear system condition. See Figure \ref{figure stretched curve} for an example of such a curve.

To choose such a curve, we need to dispatch the degree $n$ among the $g-1$ vertical loops. For a choice of dispatching $a_1+\cdots +a_{g-2}+a_\infty=n$, we have several possibilities. Choosing $k_i|a_i$, the vertical loop can make $k_i$ rounds around the vertical direction. Moreover, the $i$-th marked point has $k_i$ possible positions, and the loop has two vertices of multiplicity $\frac{a_i}{k_i}$. Concerning the unmarked vertical loop, it has $a_{\infty}$ possible positions, and contributes $k_\infty$ to $\Lambda_\Gamma^\Sigma$. Thus, we get
\begin{align*}
N_{g,(1,n)} = & \sum_{a_1+\cdots+a_{g-2}+a_\infty=n} a_\infty \sum_{k_i|a_i}k_i\left(\frac{a_i}{k_i}\right)^2 \\
 = & \sum_{a_1+\cdots+a_{g-2}+a_\infty=n} a_\infty^2\sigma_1(a_\infty)\prod_1^{g-2}a_i\sigma_1(a_i). \\
\end{align*}
We recover the result from \cite{bryan1999generating} giving the generating series of these numbers:
$$\sum_{n=1}^\infty N_{g,(1,n)}y^n = D^2G_2(y) DG_2(y)^{g-2},$$
where $G_2(y)$ is the Eisenstein series, and $D=y\frac{\dd}{\dd y}$. For the refined invariants, we get
\begin{align*}
BG_{g,(1,n)} = & \sum_{a_1+\cdots+a_{g-2}+a_\infty=n}a_\infty\sum_{k_i|a_i}k_i\left[\frac{a_i}{k_i}\right]^2. \\
\end{align*}
\end{expl}

\begin{figure}[h]
\begin{center}
\begin{tabular}{cccc}
 & \begin{tikzpicture}[line cap=round,line join=round,>=triangle 45,x=0.5cm,y=0.5cm]

\draw [line width=0.5pt] (0,0)-- ++(8,0) --++ (1,1)-- ++(0,6)-- ++(-8,0) --++ (-1,-1)-- ++(0,-6);

\draw [line width=2pt] (0,3)-- (3,3)-- ++(1,1)-- (9,4);
\draw [line width=2pt] (3,0)-- (3,3);
\draw [line width=2pt] (4,7)-- (4,4);
\draw (3,1.5) node[right] {$2$};
\draw (1.5,3) node[above] {$2$};
\draw (8.5,4) node[above] {$2$};
\draw (4,5.5) node[left] {$2$};
\draw (3.5,3.5) node[below,right] {$2$};
\end{tikzpicture} &
\begin{tikzpicture}[line cap=round,line join=round,>=triangle 45,x=0.5cm,y=0.5cm]

\draw [line width=0.5pt] (0,0)-- ++(8,0) --++ (1,1)-- ++(0,6)-- ++(-8,0) --++ (-1,-1)-- ++(0,-6);

\draw [line width=2pt] (0,3)-- (2.5,3);
\draw [line width=1pt] (2.5,3)-- ++(2,1);
\draw [line width=2pt] (4.5,4)-- (9,4);
\draw [line width=1pt] (2.5,0)-- (2.5,3);
\draw [line width=1pt] (4.5,7)-- (4.5,4);
\draw [line width=1pt] (3.5,0)-- (3.5,7);
\draw (1.25,3) node[above] {$2$};
\draw (7.25,4) node[below] {$2$};
\end{tikzpicture} &
\begin{tikzpicture}[line cap=round,line join=round,>=triangle 45,x=0.5cm,y=0.5cm]

\draw [line width=0.5pt] (0,0)-- ++(8,0) --++ (1,1)-- ++(0,6)-- ++(-8,0) --++ (-1,-1)-- ++(0,-6);

\draw [line width=1pt] (0,2.5)-- (3,2.5)-- ++(1,2)-- (9,4.5);
\draw [line width=2pt] (3,0)-- (3,2.5);
\draw [line width=2pt] (4,7)-- (4,4.5);
\draw [line width=1pt] (0,3.5)-- (9,3.5);
\draw (3,1.5) node[right] {$2$};
\draw (4,6.25) node[right] {$2$};
\end{tikzpicture} \\
$\delta_\Gamma$ & $2$ & $1$ & $1$ \\
$m_\Gamma$ & $16$ & $4$ & $4$ \\
$m^q_\Gamma$ & $\substack{q^3+2q^2+3q+4\\+3q^{-1}+2q^{-2}+q^{-3}}$ & $q+2+q^{-1}$ & $q+2+q^{-1}$ \\
$\Lambda_\Gamma^\Sigma$ & $1$ & $2$ & $2$ \\
\end{tabular}
\begin{tabular}{ccc}
 & \begin{tikzpicture}[line cap=round,line join=round,>=triangle 45,x=0.5cm,y=0.5cm]

\draw [line width=0.5pt] (0,0)-- ++(8,0) --++ (1,1)-- ++(0,6)-- ++(-8,0) --++ (-1,-1)-- ++(0,-6);

\draw [line width=1pt] (0,2.5)-- (2.5,2.5)-- ++(2,2)-- (9,4.5);
\draw [line width=1pt] (2.5,0)-- (2.5,2.5);
\draw [line width=1pt] (4.5,7)-- (4.5,4.5);
\draw [line width=1pt] (0,3.5)-- (9,3.5);
\draw [line width=1pt] (3.5,0)-- (3.5,7);
\end{tikzpicture} &
\begin{tikzpicture}[line cap=round,line join=round,>=triangle 45,x=0.5cm,y=0.5cm]

\draw [line width=0.5pt] (0,0)-- ++(8,0) --++ (1,1)-- ++(0,6)-- ++(-8,0) --++ (-1,-1)-- ++(0,-6);

\draw [line width=1pt] (0,1)-- (3,4)-- (0.5,6.5);
\draw [line width=1pt] (7,0)--++ (2,2);
\draw [line width=1pt] (8.5,0.5)-- (5,4) --++(3,3);
\draw [line width=2pt] (3,4)-- (5,4);
\draw (4,4) node[below] {$2$};
\end{tikzpicture} \\
$\delta_\Gamma$ & $1$ & $1$ \\
$m_\Gamma$ & $1$ & $4$ \\
$m^q_\Gamma$ & $1$ & $q+2+q^{-1}$ \\
$\Lambda_\Gamma^\Sigma$ & $4$ & $2$ \\
\end{tabular}

\caption{\label{figure genus 2 deg 2 2}Genus $2$ curves and their multiplicities. For each curve, there are $4$ curves up to translation in a fixed linear system.}
\end{center}
\end{figure}
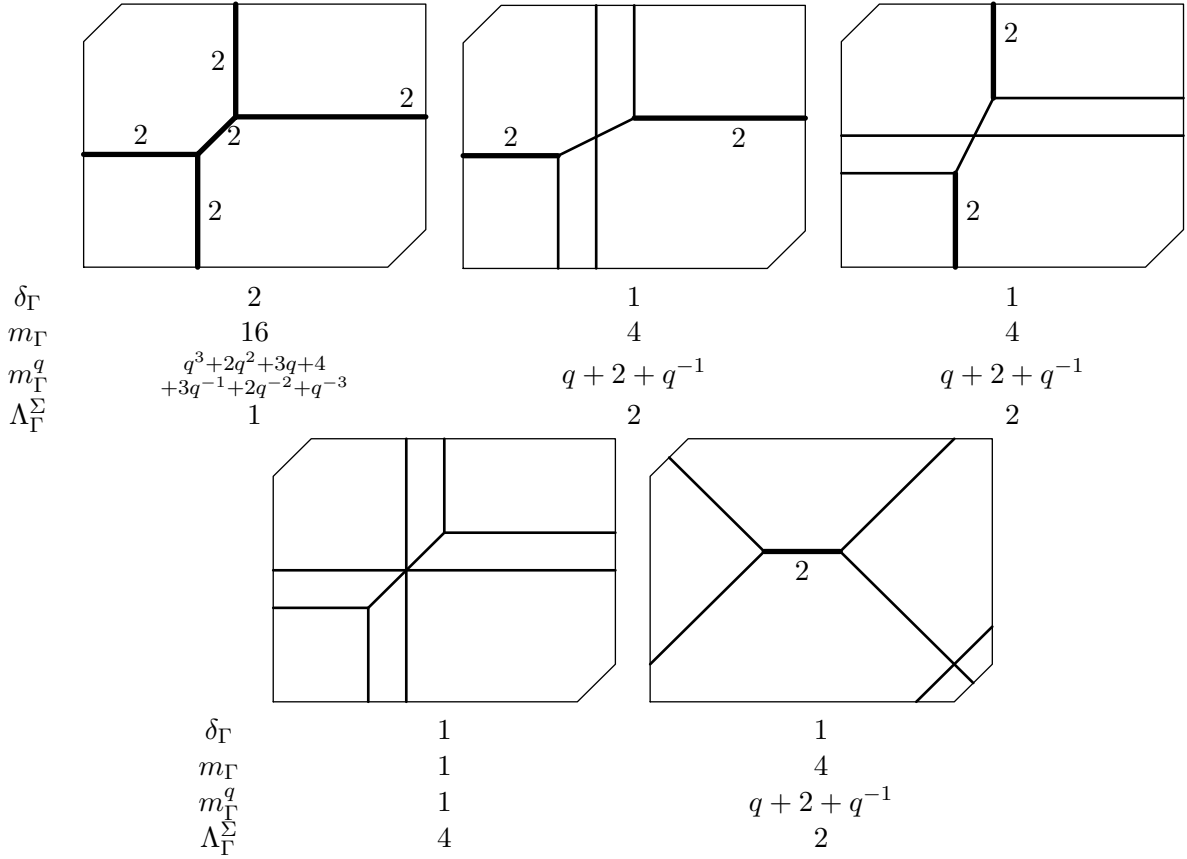

\begin{expl}
We compute the invariants for genus $2$ curves in the class $\begin{pmatrix}
2 & 0 \\
0 & 2 \\
\end{pmatrix}$. There are no marked point in this case. The curves are represented on Figure \ref{figure genus 2 deg 2 2}. Each of them has four translates. Thus, we recover
$$N_{2,(2,2)}=4(32+8+8+4+8)=240,$$
and
$$BG_{2,(2,2)}=q^3+2q^2+9q+20+9q^{-1}+2q^{-2}+q^{-3}.$$
\end{expl}

\bibliographystyle{plain}
\bibliography{biblio}

\end{document}